\documentclass[11pt,a4]{article}
\usepackage{amssymb,amsmath,amsopn,amsthm,graphicx,makeidx}
\usepackage[all,2cell]{xy} \UseAllTwocells
\usepackage{tikz}

\newtheorem{definition}{Definition}[section]

\newtheorem{prop}[definition]{Proposition}
\newtheorem{theorem}[definition]{Theorem}
\newtheorem{cor}[definition]{Corollary}

\theoremstyle{remark}

\theoremstyle{remark}

\theoremstyle{remark}

\theoremstyle{remark}
\newtheorem{example}[definition]{Example}
\theoremstyle{remark}

\theoremstyle{remark}

\theoremstyle{remark}

\theoremstyle{remark}

\theoremstyle{remark}

\theoremstyle{definition}

\theoremstyle{definition}

\theoremstyle{remark}

\renewenvironment{proof}{\noindent {\bf{Proof.}}}{\hspace*{3mm}{$\Box$}{\vspace{9pt}}}

\begin{document}

\title{Multisorted modules and their model theory}

\author{Mike Prest\\
School of Mathematics, Alan Turing Building, \\
University of Manchester,
Manchester M13 9PL, UK \\
mprest@manchester.ac.uk}

\date{}

\maketitle

\begin{abstract}
Multisorted modules, equivalently representations of quivers, equivalently additive functors on preadditive categories, encompass a wide variety of additive structures.  In addition, every module has a natural and useful multisorted extension by imaginaries.  The  model theory of multisorted modules works just as for the usual, 1-sorted modules.  A number of examples are presented, some in considerable detail.
\end{abstract}

\tableofcontents

\section{Introduction} \label{secintro}

Multisorted modules encompass a wide variety of additive structures; we illustrate this and point out that every module, even a 1-sorted one, has a natural and useful multisorted extension.  At the same time we show how to set up and use the model theory of multisorted modules.

Throughout, the emphasis is on examples and the running theme is {\em everything works}, meaning that if something is meaningful/true for ordinary modules then ditto for multisorted modules.

Throughout $K$ will denote a commutative (1-sorted) ring with 1.  It could be a field but it doesn't have to be.

\section{Multisorted modules, quiver representations and additive functors} \label{secmodrepfun}

In this section we will see that the types of object listed in the title are really the same:  they are just different ways of presenting the same information.

\vspace{4pt}

A normal, 1-sorted, module is a set with some specified algebraic structure on that set.  The idea of a multisorted module is that it is a disjoint collection of sets with algebraic structure on each of those sets but also between those sets.  The model-theoretic way to make this precise starts by specifying a multisorted language but we'll introduce that later, and rather start with the structures themselves, introducing them as, indeed defining them to be, representations of quivers.

\vspace{4pt}

A {\bf quiver} $Q$ is a directed graph:  a set $Q_0$ of vertices and a set $Q_1$ of arrows between them, multiple directed edges and loops being allowed.  We will use the quiver $A_3$ $$1 \xrightarrow{\alpha} 2 \xrightarrow{\beta} 3$$ as a running example.

We allow quivers to have infinitely many vertices, for instance $A^\infty_\infty$: $$\dots \to \bullet \to \bullet \to \bullet \to \dots$$  We will see an example of a much larger quiver in the section on Nori motives.

A $K$-{\bf representation} $V$ of a quiver $Q$ is given by associating, to each vertex $i$, a $K$-module $V_i$, and, to each arrow $i\xrightarrow{\alpha} j$, a $K$-linear map $T_\alpha:V_i \to V_j$.

Thus, a $K$-representation of the one-vertex quiver $\bullet$ with no arrow is just a $K$-module.  A $K$-representation of the quiver $\bullet \to \bullet$ is an ordered pair of $K$-modules with a homomorphism from the first to the second.  A $K$-representation of the quiver
\begin{tikzpicture} [baseline = -1mm]
\node at (0,0) {$1$}; \node at (1.1,0) {$2$};
\draw [->, thick] (0.1,0.1) to [out=45, in=135] (0.9,0.1);
\draw [->, thick] (0.1,-0.1) to [out=-45, in =-135] (0.9,-0.1);
\node at (0.5,0.4) {$\alpha$}; \node at (0.5, -0.45) {$\beta$};
\end{tikzpicture} is given by two $K$-modules and a pair of $K$-linear maps from the first to the second.  A representation of the quiver with just one vertex and a loop is a $K$-module together with a specified endomorphism, equivalently a module over the polynomial ring $K[T]$.

A {\bf morphism} $f$ from the $K$-representation $V$ of $Q$ to the $K$-representation $V'$ of the same quiver is (given by) a set $f_i:V_i\to V'_i$ ($i\in Q_0$) of $K$-linear maps such that, for each arrow $\alpha\in Q_1$, the natural diagram (below) commutes.

$\xymatrix{V_i \ar[r]^{T_\alpha} \ar[d]_{f_i} & V_j \ar[d]^{f_j} \\
V'_i \ar[r]_{T'_\alpha} & V'_j
}$

Thus we obtain the category ${\rm Rep}_KQ$ of $K$-representations of $Q$.  Like the category of modules over any ring, this is an abelian Grothendieck category, with the additive group structure on the set $(V,V')$ of morphisms from $V$ to $V'$ being given by the natural abelian group structures on the $(V_i,V'_i)$ for $i\in Q_0$

\begin{example} Take $Q$ to be $A_3$ and suppose that $K$ is a field.  Then a $K$-representation of $A_3$ is given by three $K$-vectorspaces $V_1$, $V_2$, $V_3$ and two $K$-linear maps $T_\alpha:V_1 \to V_2$ and $T_\beta : V_2 \to V_3$.  If $f=(f_1,f_2,f_3)$ and $g=(g_1,g_2,g_3)$ are morphisms from one representation to another, then the sum of these morphisms is given by $f+g=(f_1+g_1, f_2+g_2, f_3+g_3)$.

The abelian category ${\rm Rep}_KA_3$ is particularly simple because the quiver $A_3$ is of {\bf finite representation type} meaning that there are, up to isomorphism, just finitely many indecomposable\footnote{meaning direct-sum-indecomposable: a representation is such if it is not the direct sum of two non-zero representations} representations and every representation is a direct sum of copies of these indecomposables.  It is easy, using elementary linear algebra, to identify the indecomposable representations, which are as follows; we use the notation $1_K$ to denote any isomorphism of a 1-dimensional vector space to another which thereby identifies them as copies of the same object.

\begin{tikzpicture}
\node at (0,0) {$0$}; \node at (0.6,0.6) {$0$}; \node at (1.2,1.2) {$K$};
\draw [->] (0.15,0.15) to (0.45,0.45);
\draw [->] (0.75,0.75) to (1,1);

\node at (1.5,0) {$0$}; \node at (2.1,0.6) {$K$}; \node at (2.7,1.2) {$K$};
\draw [->] (1.65,0.15) to (1.9,0.4);
\draw [->] (2.3,0.8) to (2.5,1);
\node at (2.2,1.05) {$1_K$};

\node at (3,0) {$K$}; \node at (3.6,0.6) {$K$}; \node at (4.2,1.2) {$K$};
\draw [->] (3.2,0.2) to (3.4,0.4);
\draw [->] (3.8,0.8) to (4,1);
\node at (3.7,1.05) {$1_K$};
\node at (3.15,0.5) {$1_K$};

\node at (4.5,0) {$0$}; \node at (5.1,0.6) {$K$}; \node at (5.6,1.2) {$0$};
\draw [->] (4.65,0.15) to (4.9,0.4);
\draw [->] (5.3,0.8) to (5.5,1);

\node at (6,0) {$K$}; \node at (6.6,0.6) {$K$}; \node at (7.2,1.2) {$0$};
\draw [->] (6.2,0.2) to (6.4,0.4);
\draw [->] (6.8,0.8) to (7,1);
\node at (6.15,0.5) {$1_K$};

\node at (7.5,0) {$K$}; \node at (8.1,0.6) {$0$}; \node at (8.7,1.2) {$0$};
\draw [->] (7.7,0.2) to (7.9,0.4);
\draw [->] (8.3,0.8) to (8.5,1);
\end{tikzpicture}

There is a nice way of displaying these indecomposables and the morphisms between them.  It is the {\bf Auslander-Reiten quiver} for the category of representations.  This quiver has a vertex for each finite-dimensional indecomposable representation and an arrow for each member of an independent\footnote{In the sense of being linearly independent modulo the space of morphisms which have non-trivial factorisations.} set of of irreducible morphisms between indecomposables (a morphism between indecomposables is {\bf irreducible} it it has no non-trivial factorisation).  In the case of a quiver, like $A_3$, of finite representation type, every morphism between indecomposables is a linear combination of compositions of irreducible morphisms so, in such a case, this really does give a complete picture of the category of representations.  In the diagram we have expanded each vertex to show which representation is there, at the same time simplifying our notation by giving just the dimension of the vector space at each vertex of a representation.

\begin{tikzpicture}
\node at (0,0) {$1$}; \node at (0.6,0.6) {$2$}; \node at (1.2,1.2) {$3$};
\draw [->] (0.1,0.1) to (0.45,0.45);
\draw [->] (0.7,0.7) to (1.05,1.05);
\node at (0.7,1.05) {$\beta$};
\node at (0.15,0.5) {$\alpha$};

\node at (3,0) {$0$}; \node at (3.2,0.2) {$0$}; \node at (3.4,0.4) {$1$};
\node at (4.2,1.2) {$0$}; \node at (4.4,1.4) {$1$}; \node at (4.6,1.6) {$1$};
\node at (5.4,2.4) {$1$}; \node at (5.6,2.6) {$1$}; \node at (5.8,2.8) {$1$};
\node at (5.4,0) {$0$}; \node at (5.6,0.2) {$1$}; \node at (5.8,0.4) {$0$};
\node at (6.6,1.2) {$1$}; \node at (6.8,1.4) {$1$}; \node at (7,1.6) {$0$};
\node at (7.8,0) {$1$}; \node at (8,0.2) {$0$}; \node at (8.2,0.4) {$0$};

\draw [->,thick] (3.55,0.55) to (4.05,1.05);
\draw [->,thick] (4.75,1.75) to (5.25,2.25);
\draw [->,thick] (4.6,1.1) to (5.3,0.4);
\draw [->,thick] (5.8,2.3) to (6.5,1.6);
\draw [->,thick] (5.95,0.55) to (6.45,1.05);
\draw [->,thick] (7,1.1) to (7.7,0.4);

\draw [dotted, thick] (3.6,0.2) to (5.2,0.2);
\draw [dotted, thick] (6,0.2) to (7.6,0.2);
\draw [dotted, thick] (4.8,1.4) to (6.4,1.4);
\end{tikzpicture}

For example, the leftmost two maps are the obvious inclusions of \raisebox{-0.3cm}{\begin{tikzpicture} \node at (0,0) {$0$}; \node at (0.2,0.2) {$0$}; \node at (0.4,0.4) {$1$};\end{tikzpicture}} in \raisebox{-0.3cm}{\begin{tikzpicture} \node at (0,0) {$0$}; \node at (0.2,0.2) {$1$}; \node at (0.4,0.4) {$1$};\end{tikzpicture}}  and of that in \raisebox{-0.3cm}{\begin{tikzpicture} \node at (0,0) {$1$}; \node at (0.2,0.2) {$1$}; \node at (0.4,0.4) {$1$};\end{tikzpicture}}.  The inclusion of the first in the third is not irreducible since it is the composition of the other two maps.

The dotted lines in the diagram above indicate almost split sequences (also called Auslander-Reiten sequences):  these are the ``minimal'' exact sequences in the category.  In the diagram we can see three:
$$0 \to \raisebox{-0.3cm}{\begin{tikzpicture} \node at (0,0) {$0$}; \node at (0.2,0.2) {$0$}; \node at (0.4,0.4) {$1$};\end{tikzpicture}} \to \raisebox{-0.3cm}{\begin{tikzpicture} \node at (0,0) {$0$}; \node at (0.2,0.2) {$1$}; \node at (0.4,0.4) {$1$};\end{tikzpicture}} \to \raisebox{-0.3cm}{\begin{tikzpicture} \node at (0,0) {$0$}; \node at (0.2,0.2) {$1$}; \node at (0.4,0.4) {$0$};\end{tikzpicture}} \to 0$$ $$0 \to \raisebox{-0.3cm}{\begin{tikzpicture} \node at (0,0) {$0$}; \node at (0.2,0.2) {$1$}; \node at (0.4,0.4) {$1$};\end{tikzpicture}} \to \raisebox{-0.3cm}{\begin{tikzpicture} \node at (0,0) {$1$}; \node at (0.2,0.2) {$1$}; \node at (0.4,0.4) {$1$};\end{tikzpicture}} \oplus \raisebox{-0.3cm}{\begin{tikzpicture} \node at (0,0) {$0$}; \node at (0.2,0.2) {$1$}; \node at (0.4,0.4) {$0$};\end{tikzpicture}} \to \raisebox{-0.3cm}{\begin{tikzpicture} \node at (0,0) {$1$}; \node at (0.2,0.2) {$1$}; \node at (0.4,0.4) {$0$};\end{tikzpicture}} \to 0$$ $$0 \to \raisebox{-0.3cm}{\begin{tikzpicture} \node at (0,0) {$0$}; \node at (0.2,0.2) {$1$}; \node at (0.4,0.4) {$0$};\end{tikzpicture}} \to \raisebox{-0.3cm}{\begin{tikzpicture} \node at (0,0) {$1$}; \node at (0.2,0.2) {$1$}; \node at (0.4,0.4) {$0$};\end{tikzpicture}} \to \raisebox{-0.3cm}{\begin{tikzpicture} \node at (0,0) {$1$}; \node at (0.2,0.2) {$0$}; \node at (0.4,0.4) {$0$};\end{tikzpicture}} \to 0.$$
Of course we mean essentially three - for example, there extensions of $\raisebox{-0.3cm}{\begin{tikzpicture} \node at (0,0) {$0$}; \node at (0.2,0.2) {$1$}; \node at (0.4,0.4) {$0$};\end{tikzpicture}}$ by $\raisebox{-0.3cm}{\begin{tikzpicture} \node at (0,0) {$0$}; \node at (0.2,0.2) {$0$}; \node at (0.4,0.4) {$1$};\end{tikzpicture}}$ form a vector space but it is 1-dimensional.

We can also see, for example, that the morphisms from $\begin{tikzpicture} \node at (0,0) {$0$}; \node at (0.2,0.2) {$1$}; \node at (0.4,0.4) {$1$};\end{tikzpicture}$ to $\begin{tikzpicture} \node at (0,0) {$1$}; \node at (0.2,0.2) {$1$}; \node at (0.4,0.4) {$0$};\end{tikzpicture}$ form a 1-dimensional vector space since exactness of the second sequence says that going from $\begin{tikzpicture} \node at (0,0) {$0$}; \node at (0.2,0.2) {$1$}; \node at (0.4,0.4) {$1$};\end{tikzpicture}$ to $\begin{tikzpicture} \node at (0,0) {$1$}; \node at (0.2,0.2) {$1$}; \node at (0.4,0.4) {$0$};\end{tikzpicture}$ {\it via} $\begin{tikzpicture} \node at (0,0) {$1$}; \node at (0.2,0.2) {$1$}; \node at (0.4,0.4) {$1$};\end{tikzpicture}$ and {\it via} $\begin{tikzpicture} \node at (0,0) {$0$}; \node at (0.2,0.2) {$1$}; \node at (0.4,0.4) {$0$};\end{tikzpicture}$ are equivalent (up to scalar multiple).

Each representation of $A_3$ is, in a very natural sense, a 3-sorted module.  Underlying it are 3 disjoint sets - its 3 sorts - on each of which there is an algebraic structure (that of an additive group plus each element of $K$ acting as a scalar) and there are also scalars ($\alpha$, $\beta$) which move elements between sorts.
\end{example}

Once you start noticing them, you can see representations of quivers in many places.

In particular, any module over a ring $R$ can be obtained as a representation of a quiver with a single vertex $\ast$.  To do that, choose a set $(r_\lambda)_{\lambda\in \Lambda}$ of elements of $R$ which together generate $R$ as a ring.  For each $\lambda$ add a loop $\alpha_\lambda$ at $\ast$.  Let $Q$ denote the resulting quiver.  Then, given an $R$-module $M$, form a ${\mathbb Z}$-representation of $Q$ by associating to $\ast$ the underlying abelian group $M_{\mathbb Z}$ of $M$ and by associating to $\alpha_\lambda$ the scalar multiplication of $r_\lambda$ on $M$.  (If $R$ is a $K$-algebra, choose $K$-algebra generators and replace ${\mathbb Z}$ by $K$.)

While every $R$-module may be obtained this way, in general not every ${\mathbb Z}$-representation of $Q$ will be an $R$-module since there may be (polynomial) relations between the $r_\lambda$.  This naturally leads one to consider {\bf quivers with} (or {\bf bound by}) {\bf relations} and the corresponding notion of representation of a quiver with relations (a representation where those relations all hold).  In this way one can obtain every category ${\rm Mod}\mbox{-}R$ of modules as the category of representations of a quiver with relations.

Indeed, if a quiver (with relations) has just finitely many vertices then one can construct from it the corresponding path algebra (over $K$) and then the category of modules over this path algebra will be isomorphic to the category of $K$-representations of the quiver (with relations).  The {\bf path algebra} is, as a module over $K$, free on the paths in the quiver, including a ``lazy path'' at each vertex, where a path is a composable sequence of arrows.  Multiplication of paths is defined to be concatenation when possible, otherwise $0$.  For instance, in the case of $A_3$ the $K$-path algebra is the ring of upper (or lower, depending on one's convention for composition of arrows) $3\times 3$ matrices over $K$.

\begin{example} Let's use left modules with the convention that arrows compose from right to left, so $\beta\alpha$ is the path which is represented by a composition $T_\alpha$ then $T_\beta$.  Then the path algebra $KA_3$ is the matrix ring $\left( \begin{array}{ccc} K & 0 & 0 \\ K & K & 0 \\ K & K & K \end{array}\right)$.  A $K$-basis for $KA_3$ is given by the entries of the matrix $\left( \begin{array}{ccc} e_1 & 0 & 0 \\ \alpha & e_2 & 0 \\ \beta\alpha & \beta & e_3 \end{array}\right)$, where $e_i$ can be thought of as the lazy path at vertex $i$ (which acts by fixing everything at vertex $i$ and as $0$ on elements at other vertices).

Given a $K$-representation $(V_1, V_2, V_3, T_\alpha, T_\beta)$ of $A_3$, the corresponding $KA_3$-module has underlying $K$-module $V_1 \oplus V_2 \oplus V_3$.  To give the actions of the elements of $R$ it is enough to give those of the basis elements and these are just the obvious ones if we think of the $3\times 3$ matrices acting on column vectors $\left( \begin{array}{c} v_1 \\ v_2 \\ v_3 \end{array}\right)$ from the left.

In the other direction, from a left $KA_3$-module, we extract the $K$-representation of $A_3$ with $V_i =e_iM$ and the obvious (restrictions/corestrictions of multiplication by a ring element) actions for $T_\alpha$ and $T_\beta$.
\end{example}

In this example, we see that a category of 3-sorted modules - the category of representations of $A_3$ - is equivalent to a category of (ordinary) 1-sorted modules over a ring - the path algebra.  In one direction the equivalence lumps together the sorts (formally, as a direct sum, with definable components) and in the other it separates off, as various sorts, these definable components.  If there are just finitely many sorts we can do this but, once we have a quiver with infinitely many vertices, the corresponding path algebra will not have a global 1 and, arguably, viewing the representations as multi-sorted modules is more natural than trying to see them as modules over a ring without 1.

\vspace{4pt}

Now we show how additive functors on skeletally small preadditive categories are really just multisorted modules, equivalently are representations of quivers.  A category is ({\bf skeletally}) {\bf small}\footnote{Sometimes we will blur the distinction between small and skeletally small categories, since it makes no difference to the representation theory of a category.} if it has (up to isomorphism) just a set of objects and is {\bf preadditive} if it is enriched in ${\mathbb Z}$-modules, that is if each hom set has the structure of an abelian group and composition is bilinear.  (More generally we can consider categories enriched over $K$-modules.)

First we have to extend the concept of a ring to allow rings with many objects (equivalently multi-sorted rings).  Let us being looking in the other direction, by restricting the concept of a preadditive category to that of a ring.

So suppose that we have a preadditive category with just one object, $\ast$ say.  Then the endomorphisms of this object form a ring:  the preadditive structure gives the additive abelian group structure on the endomorphisms and composition of endomorphisms is the multiplication.  Conversely any ring (with 1) can be turned into such a one-object preadditive category in an obvious reversal of this process.  Thus general preadditive categories can be viewed as rings with more than one object (but, if skeletally small then, up to isomorphism, just a set of objects).  A natural example of a skeletally small (but not small) preadditive category is the category $R\mbox{-}{\rm mod}$ of finitely presented modules (that is, modules which are finitely generated and finitely related) over any ring.

If $R$ is a ring, regarded as a preadditive category with unique object $\ast_R$, then an additive functor $M$ from that category to ${\bf Ab}$ is given by an abelian group, $M(\ast_R)$, together with, for each $r\in R$, an endomorphism of that abelian group.  The additive functoriality conditions translate exactly to the condition that $M(\ast_R)$, equipped with these actions of elements of $R$, is a left $R$-module.  And conversely, every left $R$-module gives such a functor.  Similarly, natural transformations between such additive functors are exactly the $R$-homomorphisms between the corresponding modules.  That is, the category $R\mbox{-}{\rm Mod}$, of left $R$-modules, is isomorphic to the category $(R, {\bf Ab})$ of additive functors from $R$, regarded as a preadditive category, to ${\bf Ab}$.  Therefore, if ${\mathcal R}$ is any skeletally small preadditive category, then we refer to additive functors from ${\mathcal R}$ to ${\bf Ab}$ as left ${\mathcal R}$-modules, writing ${\mathcal R}\mbox{-}{\rm Mod}$ for the category they form.

For instance, if ${\mathcal R} = R\mbox{-}{\rm mod}$ is the category of finitely presented modules over some ring (or preadditive category) $R$ then the functor category $(R\mbox{-}{\rm mod}, {\bf Ab})$ may be regarded as the category of $R\mbox{-}{\rm mod}$-modules.  Thus the model theory of multisorted modules includes the model theory of additive functors on skeletally small preadditive categories.

Right $R$-modules are contravariant functors from $R$, regarded as a 1-object category, to ${\bf Ab}$, equivalently as left modules over the opposite ring $R^{\rm op}$, and we extend the terminology and notation to multisorted modules in the obvious ways.

\begin{example}  What is a finitely presented multisorted module?  A module $M$ over a 1-sorted ring $R$ is finitely presented if it has a projective presentation of the form $_RR^m \to _RR^n \to M \to 0$:  an exact sequence where $_RR$ denotes the ring regarded as a left module over itself.  That module is the image of the object $\ast_R$ under the (contravariant) Yoneda embedding of $R$, regarded as above as a 1-object category, to $R\mbox{-}{\rm Mod}$, that is, it is the representable functor $(\ast_R,-)$.  In the general case there are many objects in ${\mathcal R}$, so the finitely generated free modules are replaced by any finite direct sum of the form $(p_i,-)$ with the $p_i$ objects of ${\mathcal R}$, giving the definition of finitely presented ${\mathcal R}$-module.

Equivalently, an ${\mathcal R}$-module $M$ is {\bf finitely presented} if there is an exact sequence $Q \to P \to M \to 0$ where $P$ and $Q$ are finitely generated projective ${\mathcal R}$-modules (these being the direct summands of direct sums of representable functors $(p,-)$).
\end{example}

\vspace{4pt}

Finally in this section, we illustrate how to go directly from representations of quivers to functors on preadditive categories.

\begin{example} The quiver $A_3$ is not a category but there is an obvious category, namely the free one, that we can build from it, as from any quiver $Q$.  The objects of that category, which we denote $\overrightarrow{Q}$, are the vertices of $Q$ and the morphisms from $i$ to $j$ are the {\bf path}s from $i$ to $j$ (compatible concatenations of arrows where the head of one is at the tail of the next, working from right to left), including a ``lazy path'' at each vertex $i$ (the identity at $i$).  Composition of morphisms is given by concatenation where compatible and 0 where not.

In the case of $A_3$ this adds the identity map at each vertex and the path $\beta\alpha$ from $1$ to $3$.

The free preadditive category ${\mathbb Z}\overrightarrow{Q}$ on $Q$ is obtained by replacing each set $\overrightarrow{Q}(i,j)$ of morphisms by the free abelian group on that set, and extending the composition of morphisms by requiring that it be bilinear.

In our example, we see that ${\mathbb Z}\overrightarrow{A_3}$ is such that $(i,j)$ is ${\mathbb Z}$ if $i\leq j$ and is $0$ otherwise.

Of course we can replace ${\mathbb Z}$ by any commutative ring $K$ and then we have the free $K\mbox{-}{\rm Mod}$-enriched preadditive category $K\overrightarrow{Q}$ on $Q$.  Note that this is not identical to the $K$-path algebra $KQ$:  the latter, as a preadditive category, has just one object, whereas $K\overrightarrow{A_3}$ has three.  But they are almost the same as categories, in that they have the same idempotent-splitting additive completion.  The free additive completion of $KQ$ is obtained by formally adding finite direct sums of copies of the objects and the obvious morphisms; then the idempotent-splitting completion adds kernels and cokernels of idempotent maps to that. (Considering the Yoneda embedding which takes each object to the corresponding representable functor, that is left module, one sees that the additive completion is the opposite of the category of finitely generated free $KQ$-modules and the idempotent-splitting completion of that is the opposite of the category of finitely generated projective $KQ$-modules.)  One can easily see that starting with $K\overrightarrow{Q}$ leads to the same result.  In particular, the categories have the same representation theory:  $KQ\mbox{-}{\rm Mod} \simeq K\overrightarrow{Q}\mbox{-}{\rm Mod}$.
\end{example}

\section{Setting up linear algebra in multisorted modules}

Linear algebra over a field (more generally, over a von Neumann regular ring) is about systems of linear equations.  But over other rings, projections of solution sets of systems of linear equations are in general not solution sets of systems of linear equations, and existential quantifiers are needed to define them.  Consider, for example, the relation of divisibility over $R={\mathbb Z}$:  the condition $3|x$ is obtained from the linear equation $x-3y=0$ by projecting out the second coordinate:  $\exists y\, (x-3y=0)$.  There is no way of avoiding that quantifier so we must accept that linear algebra over general rings necessarily involves existentially quantified systems of linear equations.

Fortunately, such expresssions - existentially quantified systems of linear equations - are familiar from model theory where they are treated as mathematical objects, referred to as positive primitive (pp for short) formulas\footnote{or ``regular'' formulas in categorical model theory}.  Model theory has the tools to handle these and their solution sets.

There are many accounts of this, so here I recall only what I need.

A {\bf pp formula} is an existentially quantified system of linear equations.  A system of linear equations for left $R$-modules has the form $H\overline{x}=0$ where $H$ is a rectangular matrix with entries from $R$ and $\overline{x}$ is a (column) vector of variables.  So a pp formula has the form $\exists \overline{y} \,\, G(\overline{x} \, \overline{y})=0$ for some matrix $G$ over $R$.  We can denote such a formula by $\phi$ or, showing its free=unquantified variables, $\phi(\overline{x})$.  The {\bf solution set}, $\phi(M)$, of $\phi$ in a right $R$-module $M$ is $\{ \overline{a}\in M^n: \exists \overline{b}\in M^m \mbox{ such that } G(\overline{a}\, \overline{b})=0\}$ where $n$, $m$ are the lengths of the tuples $\overline{x}$, $\overline{y}$ respectively.  This is a {\bf pp-definable subgroup} of $M$ (more precisely, a subgroup of $M^n$ pp-definable in $M$).  The intersection and then sum of two such pp-definable subgroups are again pp-definable.

Such solution sets are preserved by morphisms.

\begin{prop}  Let $f:M\to N$ be a morphism of $R$-modules.  Then $f\phi(M) \leq \phi(N)$.

In particular every subgroup of $M^n$ pp-definable in $M$ is an ${\rm End}(M)$-submodule of $M^n$, where ${\rm End}(M)$ acts diagonally on $M^n$.
\end{prop}

Therefore each pp formula $\phi$ defines an additive functor from ${\rm Mod}\mbox{-}R$ to ${\bf Ab}$ (equivalently to $K\mbox{-}{\rm Mod}$ if $R$ is a $K$-algebra), which we denote by $F_\phi$, given on objects by $F_\phi(M)=\phi(M)$.

\begin{prop}  Suppose that $M = \varinjlim_\lambda \, M_\lambda$ is a direct limit (= directed colimit) of modules and let $\phi$ be a pp formula.  Then $\phi(M) = \varinjlim_\lambda \, \phi(M_\lambda)$.

In particular, each functor of the form $M\mapsto \phi(M)$ with $\phi$ pp is determined by its action on finitely presented modules.
\end{prop}

The second statement follows since every $R$-module is a direct limit of finitely presented $R$-modules.  We will use this in Section \ref{secthree}.

\section{Multisorted modules as structures}

We set up the model theory of multisorted modules in the obvious way.  If we're starting with a quiver $Q$, then we introduce a sort $\sigma_i$ for each vertex $i\in Q_0$ and a function symbol $f_\alpha$ from sort $\sigma_i$ to $\sigma_j$ for each arrow $\alpha:i \to j$ in $Q$.  We want our structures to be additive, so we equip each sort with a symbol $+$ (that is, $+_i$) for the addition and a constant symbol $0$ (that is $0_i$) for the zero element of that sort.  If we want to consider $K$-representations then we should add, for each $i$ and each $\lambda \in K$, a 1-ary function symbol to express multiplication by $\lambda$ on sort $\sigma_i$.

With that language we can then write down the requirements that each sort $\sigma_i$ have the structure of a $K$-module, that each function $f_\alpha$ be $K$-linear and, if we started with a quiver with relations, that all the imposed relations are satisfied.  In this way the category of $K$-representations of $Q$ (with any relations) becomes the category of models of a set of equations in the language we just set up.  This gives us the view of multisorted modules as structures for some language, where the {\bf sorts} of a multisorted module, that is, a $K$-representation $M$ of $Q$, are the $K$-modules $M_i$ for $i\in Q_0$.

If the quiver has just one vertex then, of course, we get the usual language for modules over a $K$-algebra (that $K$-algebra being determined by the arrows of $Q$ and any imposed relations).  But this multisorted set-up allows us to encompass many more examples, some of which are described in the next, and later, sections.  The main methodological point is that the model theory of multisorted modules works exactly as does that for 1-sorted modules.  Of course, one has to say some things slightly differently because of there being more than one sort and one has to bear in mind that any single formula refers to only finitely many sorts of the structure (variables are sorted, as are the function and constant symbols).  So a pp-definable subgroup of a multisorted module will be a subgroup of a product $M_{i_1} \times \dots \times M_{i_n}$ of finitely many sorts, in particular, the whole module is not a definable set if there are infinitely many vertices in $Q$.

\vspace{4pt}

What changes would be needed if we were to start with a small preadditive category instead of a quiver, that is, if we were to interpret ``multisorted module'' as meaning a $K$-linear functor from such a category to $K\mbox{-}{\rm Mod}$, equivalently if we were to start with a module over a multisorted ring?  Noting that a category can be seen as a quiver, what is extra here is the addition in each hom-set, so we will have more axioms to write down, expressing relations between arrows of that quiver which involve addition as well as composition.  Otherwise everything is as before.

\vspace{4pt}

The construction above gives the ``covariant'' language but sometimes, see \ref{exfinacc} below, it is more natural to consider right modules = contravariant functors, in which case, for each arrow $\alpha:i \to j$ of the quiver (or category) we would introduce a function symbol from sort $\sigma_j$ to $\sigma_i$.

\vspace{4pt}

What about categories, such as $R\mbox{-}{\rm mod}$, that are skeletally small but not small?  Then we just choose some full subcategory which contains at least one copy of each object and apply the above.

\section{Examples of multisorted modules}

\begin{example}  {\bf Chain complexes of modules:}  Suppose that $R$ is any ring.  A {\bf chain complex} of $R$-modules is a sequence $(M_i)_{i\in {\mathbb Z}}$ of $R$-modules and morphisms $(d_i:M_i \to M_{i-1})_i$ such that $d_{i-1}d_i=0$ for every $i$.  The usual way of writing this is $$\dots \to M_2 \xrightarrow{d_2} M_1 \xrightarrow{d_1} M_0 \xrightarrow{d_0} M_{-1} \xrightarrow{d_{-1}} M_{-2} \to \dots,$$ which makes it clear that this is naturally an $R$-representation of the quiver $A_\infty^\infty = \dots \to \cdot \to \cdot \to \cdot \to \dots$ (the notion of $K$-representation of a quiver doesn't require the ring $K$ to be commutative).

Thus complexes are multisorted modules and the language for them has sorts indexed by ${\mathbb Z}$.  The language, as well as being able to express the $R$-module structure in each sort, has a function symbol for each $i$ to express the differential $d_i$.  The conditions $d_{i-1}d_{i}=0$ give axioms $\forall x_{i} \, d_{i-1}d_{i}x_{i} =0_{i-2}$, where subscripts on the variable and the constant symbol indicate their sorts.  Each of these axioms says that a certain quotient of pp-definable groups is $0$, in this case, the quotient $(x_i=x_i)/(d_{i-1}d_{i}x_{i} =0_{i-2})$.  These infinitely many pp conditions cut out the chain complexes as an axiomatisable (indeed, ``definable" in the sense we discuss in Section \ref{secloc}) subclass of the category of $R$-representations of $A_\infty^\infty$.

Similarly, the exact sequences are obtained as those on which all the pp-pairs $(d_ix_i=0)/(\exists x_{i+1}\, (x_i=d_{i+1}x_{i+1}))$ are ``closed'' (that is, $0$).  Such pp-pairs, that is the various homologies of such a complex, will appear as new sorts in an enriched language which we define in the next section.  An example of a subcategory that cannot be defined by only finitely many formulas (saying that certain pp-pairs are ``closed") is that consisting of the complexes concentrated at a given vertex $i$ (that, which are $0$ on every other vertex).
\end{example}

\begin{example} \label{exfinacc} {\bf Finitely accessible categories:}  A category ${\mathcal C}$ is {\bf finitely accessible} if it has direct limits, if its finitely presented objects form a set up to isomorphism and if every object of ${\mathcal C}$ is a direct limit of finitely presented objects.  An object $X\in {\mathcal C}$ is {\bf finitely presented} if the covariant representable functor $(X,-): {\mathcal C} \to {\bf Set}$ commutes with direct limits.  This does coincide with the more familiar definition - as being finitely generated and finitely related - when the latter makes sense.  The category of modules over a ring (or over a small preadditive category) is an example as is, for instance, the category of groups.  A standard reference is \cite{AdRo}.

There is a natural language for a finitely accessible (additive) category ${\mathcal C}$, based on its full, skeletally small, subcategory ${\mathcal C}^{\rm fp}$ of finitely presented objects.  Namely the language which was described in the previous section, based on a skeletally small category, with a sort for each finitely presented object (in some small version of ${\mathcal C}^{\rm fp}$), {\em but} the contravariant version.  Contravariance is natural because each object $C$ of ${\mathcal C}$ is naturally a structure for the contravariant language.  Namely, the set of objects of $C$ of sort $X\in {\mathcal C}^{\rm fp}$ is defined to be the set, $(X,C)$, of morphisms from $X$ to $C$.  Then a morphism $\alpha:X\to Y$ induces, by composition, a map from $(Y,C)$ to $(X,C)$, hence the contravariance.  Of course, here we are interested in finitely accessible categories which are additive.  It also turns out that, in order to have model theory work essentially as in modules, we should require ${\mathcal C}$ to have products, but that is enough - every finitely accessible category with products is definable in the sense of Section \ref{secloc}.

Note that, for ${\mathcal C} = R\mbox{-}{\rm Mod}$ the language we obtain this way is richer than the usual 1-sorted language since it has at least one sort for every isomorphism type of finitely presented module $A$.  Recall that $(_RR,M)$ is naturally isomorphic to $M$ for $M\in R\mbox{-}{\rm Mod}$ (associate to $f:_RR \to M$ the element $f1$), so this language contains the usual 1-sorted language.  But the language has many more sorts though, as we will see next, they are all definable in terms of that ``home sort''.

For instance $(_RR^n,M) \simeq M^n$, so the $n$-tuples from a module are the elements of one of its sorts.  In fact, each of the new sorts $(A,-)$ is (non-canonically) isomorphic, given a choice of a finite sequence $\overline{a}$ of generators for $A$, to a definable subset of the sort $(-)^n$ where $n$ is the length of $\overline{a}$.  The identification takes a morphism $f:A\to M$ to the image $f\overline{a}$.  The formula defining the set of possible images is quantifier-free, consisting of the conjunction of any finitely many generators for the $R$-linear relations between the entries of $\overline{a}$.  So the enrichment of a module to this language is part of the enrichment obtained by adding all imaginary sorts (in the sense that we will discuss in Section \ref{secloc}).  Indeed it consists of the enrichment by pp-pairs of the form $\theta(\overline{x})/(\overline{x} = \overline{0})$  where $\theta$ is a quantifierfree pp formula.
\end{example}

\begin{example} {\bf Comodules over a coalgebra:}  Let $C$ be a $K$-coalgebra; then the structure map of a $C$-comodule $M$ has the form $\rho: M\to M\otimes_K C$.  At the outset, it was not at all clear how, or whether, such structures could be treated model-theoretically.  However, if $K$ is a field, then the category $C\mbox{-}{\rm Comod}$ of $C$-comodules is finitely accessible, indeed is locally finitely presented, which implies finitely accessible.  More generally, this is true if $K$ is noetherian and $C$ is a projective $K$-module, in which case $C\mbox{-}{\rm Comod}$ is equivalent to the category of ``rational" modules over the dual algebra and it is the case that every rational module over that dual algebra is a direct limit of finitely presented rational modules (see \cite{ReyThes}).

We conclude then, from the previous example, that we can develop a model theory of comodules as multisorted modules - see \cite{CrivPreRey}, \cite{ReyThes}, which also have references for the relevant background - with sorts corresponding to the finitely presented rational modules over the dual algebra.
\end{example}

\begin{example}  {\bf Presheaves and sheaves:}  If $(X, {\mathcal O}_X)$ is any ringed space then the category of presheaves over this space is locally finitely presented (see, e.g., \cite[\S I.3]{SGA4}), hence finitely accessible, so has a natural model theory based on the finitely presented presheaves.  The finitely presented presheaves are the extensions by $0$ of the restrictions of the structure sheaf ${\mathcal O}_X$ to the various open subsets.  If the space is noetherian then the category of sheaves is a nice Gabriel localisation of the category of presheaves and it follows that it also is locally finitely presented (for a more general result see \cite{Bridge}).  However, over general ringed spaces, the category of sheaves need not be finitely accessible (see, e.g., \cite[\S 16.3.4]{PreNBK}).

For some model theory of sheaves based on this language, see \cite{PRPya}.
\end{example}

\begin{example} {\bf Locally coherent sheaves:}  If $(X, {\mathcal O}_X)$ is a nice enough scheme (for instance, if it is compact and quasi-separated \cite[6.9.12]{EGA}), then the category of quasicoherent sheaves over it is locally finitely presented, hence these sheaves may be treated as multisorted modules, with the coherent sheaves, these being the finitely presented quasicoherent sheaves, labelling the sorts.
\end{example}

\section{Adding new sorts} \label{seceq+}

The imaginaries construction in model theory adds definable or, more generally, interpretable sets as new sorts to a language.  In the additive context, all sorts should have an induced abelian-group structure, and that forces us (see \cite[2.1]{BurPre1}) to restrict to adding sorts which are defined by pp formulas.

Given a language ${\mathcal L}$ for, possibly already multisorted, modules, we add, as new sorts, the pairs, denoted $\phi/\psi$, of pp formulas where $\phi$, $\psi$ are in the same free variables and, $\psi$ implies $\phi$ in the sense that $\psi(M) \leq \phi(M)$ for every module $M$.  We add symbols $+, 0$ in each new sort to express the induced abelian-group structure.  We also add, for each pair, $\phi/\psi$ and $\phi'/\psi'$ of sorts, a function symbol for every {\bf pp-definable function} from $\phi/\psi$ to $\phi'/\psi'$.  By that we mean a function which is given by a pp-definable relation from $\phi$ to $\phi'$ which well-defines a functional and total relation from $\phi(M)/\psi(M)$ to $\phi'(M)/\psi'(M)$ for every module $M$.\footnote{It is not necessary to check in every module $M$: the condition that one pp formula imply another is equivalent to a simple algebraic condition on the matrices used to express the formulas, see \cite[1.1.13]{PreNBK}.}  We denote this new language by ${\mathcal L}^{\rm eq+}$.  Every module $M$ is naturally enriched to an ${\mathcal L}^{\rm eq+}$-structure, denoted $M^{\rm eq+}$, by taking the collection of all groups $\phi(M)/\psi(M)$ and all pp-definable maps between them - this is naturally an ${\mathcal L}^{\rm eq+}$-structure.

Note that an additive ${\mathcal L}^{\rm eq+}$-structure is nothing more than an additive functor from the category, which we denote ${\mathbb L}^{\rm eq+}$, of pp-sorts and pp-definable functions between them.  We refer to this as the {\bf category of pp-pairs} or {\bf pp-imaginaries category} and, in the case that we started with a language for left ${\mathcal R}$-modules, we denote it $_{\mathcal R}{\mathbb L}^{\rm eq+}_{\mathcal R}$.  This category is, in fact, abelian (\cite[\S 1]{HerzDual}), and we will look at the significance of that later.  We will also compute some examples of these categories in Sections \ref{secA3}, \ref{secA3again} and \ref{sectensor}.

It is easy to check that the class of ${\mathcal L}^{\rm eq+}$-structures which arise as $M^{\rm eq+}$ for some module $M$ is axiomatisable (indeed definable in the sense considered in Section \ref{secloc}) and that $M\mapsto M^{\rm eq+}$ is an equivalence between the original module category and the full subcategory, of the category of all ${\mathbb L}^{\rm eq+}$-modules, on these structures.  So, in some sense, in replacing a module $M$ by $M^{\rm eq+}$, nothing has changed:  we still have the same objects and morphisms but each has been replaced by a highly-enriched version.

Just to be clear, there will in general be many ${\mathbb L}^{\rm eq+}$-modules which are not of the form $M^{\rm eq+}$; indeed the latter are those which, regarded as functors on ${\mathbb L}^{\rm eq+}$, are {\bf exact} (take exact sequences to exact sequences), see \ref{funD} below.

We saw already in Example \ref{exfinacc} how the modules over a ring become structures for a richer language which has a sort for each finitely presented module.  But that enriched language is just a part of the full pp-imaginaries language.  In fact, those sorts - the pairs $\theta/0$ where $\theta$ is a system of linear equations, are exactly the projective objects in the category ${\mathbb L}^{\rm eq+}$ (and, more generally, the pp formulas $\phi$, regarded as pp-pairs $\phi\overline{x}/\overline{x}=\overline{0}$, are the objects of projective dimension  $\leq 1$, see \cite[10.2.14]{PreNBK}).

This operation of forming the pp-imaginaries category is idempotent in the sense that $(M^{\rm eq+})^{\rm eq+}$ is naturally isomorphic to $M^{\rm eq+}$ (model-theoretically, this is clear; an algebraic reason is remarked in \cite[\S 5]{PreAxtFlat}).

\section{Three categories}\label{secthree}

Associated to any ring or skeletally small preadditive category $R$ are the following three skeletally small abelian categories:

\noindent $\bullet$ the category $_R{\mathbb L}^{\rm eq+}$ of pp imaginaries for left $R$-modules;

\noindent $\bullet$ the category $(R\mbox{-}{\rm mod}, {\bf Ab})^{\rm fp}$ of finitely presented additive functors on finitely presented $R$-modules - this is equivalent to $(R\mbox{-}{\rm mod}, K\mbox{-}{\rm Mod})^{\rm fp}$ if $R$ is a $K$-algebra with each element of $K$ acting centrally;

\noindent $\bullet$ the {\bf free abelian category} on $R$ - Freyd \cite{FreydLJ}, see also \cite{Adel}, showed that there is an embedding $R\to {\rm Ab}(R)$ of $R$ into an abelian category which has the following universal property:  for every additive functor $M:R\to {\mathcal A}$, where ${\mathcal A}$ is an abelian category, there is a unique-to-natural-equivalence extension of $M$ to an exact functor $\widetilde{M}$ making the following diagram commute.

$\xymatrix{R \ar[r] \ar[dr]_M & {\rm Ab}(R) \ar[d]^{\widetilde{M}} \\ & {\mathcal A}
}$

It follows that the free abelian category ${\rm Ab}(R)$ on $R$ is unique up to natural equivalence.  In fact, these categories all are equivalent.

\begin{theorem} \label{3cats}  For any ring or small preadditive category $R$, there are natural equivalences  $_R{\mathbb L}^{\rm eq+} \simeq (R\mbox{-}{\rm mod}, {\bf Ab})^{\rm fp} \simeq {\rm Ab}(R)$.  Furthermore, with reference to the diagram above, $\widetilde{M} = M^{\rm eq+}$.
\end{theorem}

The second statement is explained by thinking of $M:R\to {\mathcal A}$ as an $R$-module (taking values in ${\mathcal A}$) and the domain of $\widetilde{M}$ as being the category $_R{\mathbb L}^{\rm eq+}$ of pp-imaginaries for $R$-modules.  Then $\widetilde{M}$ is equivalent to the functor which assigns to each pp-sort $\phi/\psi$ the group $\phi(M)/\psi(M)$ and similarly for pp-definable maps (for both are exact extensions of $M$).

If we think of the domain of $\widetilde{M}$ as the functor category $(R\mbox{-}{\rm mod}, {\bf Ab})^{\rm fp}$ then $\widetilde{M}$ is ${\rm ev_M}$ - evaluation-at-$M$ - which takes a functor $F\in (R\mbox{-}{\rm mod}, {\bf Ab})^{\rm fp}$ to its value $FM$.  At least, if $M$ is finitely presented then that makes sense since the functors are functors on finitely presented modules.  For general $M$, we use that each functor $F$ on finitely presented modules has an essentially unique extension which commutes with direct limits, to a functor $\overrightarrow{F}$ on all modules.  That functor is (well-)defined on a module $M$ by representing $M$ as a direct limit of finitely presented modules, applying the given functor $F$ to these and taking the direct limit of the resulting diagram of abelian groups (or objects of the Ind-completion of ${\mathcal A}$ in the general case).  So we should write ${\rm ev}_M:F\to \overrightarrow{F}(M)$:  in this way we may regard the functors in $(R\mbox{-}{\rm mod}, {\bf Ab})^{\rm fp}$ as acting on the whole of $R\mbox{-}{\rm Mod}$.

These are two complementary views of what $M^{\rm eq+}$ is:  as the functor just described, alternatively as the ``image'' of that functor - literally as the collection of images of the objects of the category and all the image maps between these - that is precisely $M^{\rm eq+}$ regarded as an $_R{\mathbb L}^{\rm eq+}$-structure.

\section{An example: $A_3$}\label{secA3}

We illustrate the ideas and results above by computing the free abelian category on the path algebra $R=KA_3$ where $K$ is a field.  We will compute the objects of that category as functors on finitely presented $R$-modules and as pp-pairs.

In fact, this example is particularly simple, first in that $R$ has finite representation type.  That implies that the functor category $(R\mbox{-}{\rm mod}, {\bf Ab})^{\rm fp}$ is actually equivalent to the category of modules over the {\bf Auslander algebra} of $R$.  That is $S={\rm End}(_RM_0)$ where $M_0$ is the direct sum of one copy of each indecomposable $R$-module.  If we regard $S$ as acting on the left of $R$-modules, then the functor category is equivalent to the category of left $S$-modules.  We will prove a somewhat more general result in Section \ref{secA3again} which will be useful in computing localisations of this functor category.

The further simplification in this example is that the Auslander algebra will itself prove to be of finite representation type (that need not be the case in general, even if the ring $R$ is of finite representation type).  It is a fairly short exercise to compute all the indecomposable functors and the morphisms between them.  Since $S$ is an {\bf artin algebra} (finitely generated as a module over an artinian centre), Auslander-Reiten theory can be used to compute its modules, just as for $R=KA_3$.

One may present $S$ as a subring of the ring of $6\times 6$ matrices over $K$, because there are $6$ indecomposable $R$-modules and the morphisms between any pair of indecomposable $R$-modules form a $\leq 1$-dimensional space over $K$.  Alternatively, $S$ may be represented as the $K$-path algebra of the Auslander-Reiten quiver of $R$, including the relations which come from the three Auslander-Reiten sequences in that quiver.  We will use the latter representation since then the $S$-modules may be given as representations of that quiver. For this particular algebra $S$, it is the case that each indecomposable is determined by the dimensions of the $K$-vector spaces at the 6 vertices (of the AR-quiver of $R$), so it is enough, in describing the indecomposables, just to give these dimensions.  The maps within each representation are the ``obvious'' ones:  each arrow of the Auslander-Reiten quiver of $R$ being represented by an isomorphism if possible, the zero map otherwise and, of course, all relations in the Auslander-Reiten quiver must be satisfied.  For example the picture

\begin{tikzpicture}
\node at (6.3,0) {$0$}; \node at (6.5,0.2) {$1$}; \node at (6.7,0.4) {$1$};
\node at (6.7,0) {$1$}; \node at (6.9,0.2) {$1$}; \node at (7.1,0) {$0$};
\end{tikzpicture}

\noindent means the representation (equivalently, $S$-module)

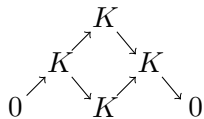
\begin{tikzpicture}
\node at (0,0) {$0$}; \node at (0.6,0.6) {$K$}; \node at (1.2,1.2) {$K$};
\node at (1.2,0) {$K$}; \node at (1.8,0.6) {$K$}; \node at (2.4,0) {$0$};

\draw [->] (0.15,0.15) to (0.4,0.4);
\draw [->] (0.75,0.75) to (1,1);
\draw [->] (0.75,0.4) to (1,0.1);
\draw [->] (1.35,0.15) to (1.6,0.4);
\draw [->] (1.35,1) to (1.6,0.7);
\draw [->] (1.95,0.4) to (2.2,0.1);
\end{tikzpicture}

\noindent where each arrow shown is $0$ if it must be and otherwise an isomorphism (``the identity'') and where, necessarily, the composition of the top two arrows is the negative of the composition of the lower two (reflecting one of the three Auslander-Reiten sequences in the AR-quiver of $R$ which was seen in Section \ref{secmodrepfun}).

Here, then, is the Auslander-Reiten quiver of $S$.  It is reasonable to assert that this gives a complete picture of the category of all $S$-modules:  it shows all indecomposable $S$-modules (and every $S$-module is a direct sum of these).  Furthermore, every morphism between indecomposable $S$-modules is a $K$-linear combination of compositions of the irreducible maps (the maps shown on the AR-quiver).  In fact, in this very simple case, every map is just a scalar multiple of compositions of irreducible maps.  Then every morphism between $S$-modules can be given as a matrix of morphisms between its indecomposable components.  So all this information is essentially in the following quiver with relations (those being given by the Auslander-Reiten sequences)

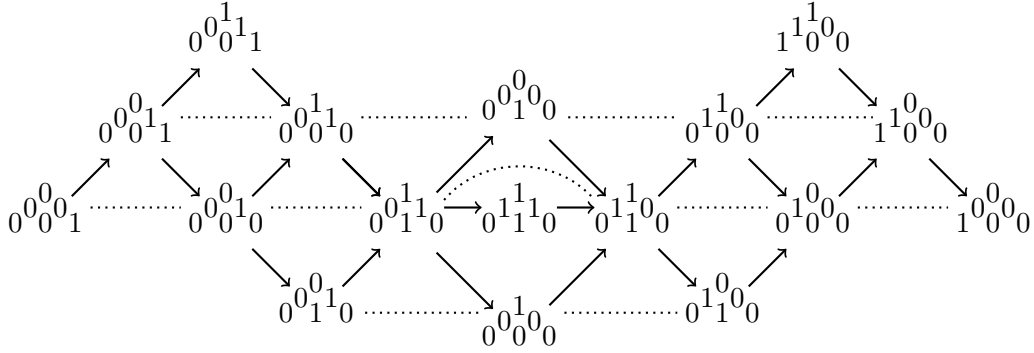
\begin{tikzpicture}

\node at (0,0) {$0$}; \node at (0.2,0.2) {$0$}; \node at (0.4,0.4) {$0$};
\node at (0.4,0) {$0$}; \node at (0.6,0.2) {$0$}; \node at (0.8,0) {$1$};
\draw [->,thick] (0.75,0.35) to (1.25,0.85);
\draw [dotted, thick] (1,0.2) to (2.2,0.2);

\node at (1.2,1.2) {$0$}; \node at (1.4,1.4) {$0$}; \node at (1.6,1.6) {$0$};
\node at (1.6,1.2) {$0$}; \node at (1.8,1.4) {$1$}; \node at (2,1.2) {$1$};
\draw [->,thick] (1.95,1.55) to (2.45,2.05);
\draw [->,thick] (1.95,0.85) to (2.45,0.35);
\draw [dotted, thick] (2.2,1.4) to (3.4,1.4);

\node at (2.4,2.4) {$0$}; \node at (2.6,2.6) {$0$}; \node at (2.8,2.8) {$1$};
\node at (2.8,2.4) {$0$}; \node at (3,2.6) {$1$}; \node at (3.2,2.4) {$1$};
\draw [->,thick] (3.15,2.05) to (3.65,1.55);

\node at (2.4,0) {$0$}; \node at (2.6,0.2) {$0$}; \node at (2.8,0.4) {$0$};
\node at (2.8,0) {$0$}; \node at (3,0.2) {$1$}; \node at (3.2,0) {$0$};
\draw [->,thick] (3.15,0.35) to (3.65,0.85);
\draw [->,thick] (3.15,-0.35) to (3.65,-0.85);
\draw [dotted, thick] (3.4,0.2) to (4.6,0.2);

\node at (3.6,1.2) {$0$}; \node at (3.8,1.4) {$0$}; \node at (4,1.6) {$1$};
\node at (4,1.2) {$0$}; \node at (4.2,1.4) {$1$}; \node at (4.4,1.2) {$0$};
\draw [->,thick] (4.35,0.85) to (4.85,0.35);
\draw [dotted, thick] (4.6,1.4) to (6,1.4);

\node at (3.6,-1.2) {$0$}; \node at (3.8,-1) {$0$}; \node at (4,-0.8) {$0$};
\node at (4,-1.2) {$1$}; \node at (4.2,-1) {$1$}; \node at (4.4,-1.2) {$0$};
\draw [->,thick] (4.35,-0.85) to (4.85,-0.35);
\draw [dotted,thick] (4.65,-1.2) to (6.2,-1.2);

\node at (4.8,0) {$0$}; \node at (5,0.2) {$0$}; \node at (5.2,0.4) {$1$};
\node at (5.2,0) {$1$}; \node at (5.4,0.2) {$1$}; \node at (5.6,0) {$0$};
\draw [->,thick] (5.6,0.4) to (6.35,1.15);
\draw [->,thick] (5.6,-0.4) to (6.35,-1.15);
\draw [->,thick] (5.7,0.2) to (6.2,0.2);
\draw [dotted,thick] (5.7,0.3) to [out=50, in=130] (7.7,0.3);

\node at (6.3,0) {$0$}; \node at (6.5,0.2) {$1$}; \node at (6.7,0.4) {$1$};
\node at (6.7,0) {$1$}; \node at (6.9,0.2) {$1$}; \node at (7.1,0) {$0$};
\draw [->,thick] (7.2,0.2) to (7.7,0.2);
\draw [->,thick] (4.35,0.85) to (4.85,0.35);
\draw [dotted,thick] (7.35,1.4) to (8.8,1.4);

\node at (6.3,1.5) {$0$}; \node at (6.5,1.7) {$0$}; \node at (6.7,1.9) {$0$};
\node at (6.7,1.5) {$1$}; \node at (6.9,1.7) {$0$}; \node at (7.1,1.5) {$0$};
\draw [->,thick] (7.1,1.1) to (7.85,0.35);

\node at (6.3,-1.5) {$0$}; \node at (6.5,-1.3) {$0$}; \node at (6.7,-1.1) {$1$};
\node at (6.7,-1.5) {$0$}; \node at (6.9,-1.3) {$0$}; \node at (7.1,-1.5) {$0$};
\draw [->,thick] (7.1,-1.1) to (7.85,-0.35);
\draw [dotted, thick] (7.3,-1.2) to (8.8,-1.2);

\node at (7.8,0) {$0$}; \node at (8,0.2) {$1$}; \node at (8.2,0.4) {$1$};
\node at (8.2,0) {$1$}; \node at (8.4,0.2) {$0$}; \node at (8.6,0) {$0$};
\draw [->,thick] (8.55,0.35) to (9.05,0.85);
\draw [->,thick] (8.55,-0.35) to (9.05,-0.85);
\draw [dotted, thick] (8.8,0.2) to (10,0.2);

\node at (9,1.2) {$0$}; \node at (9.2,1.4) {$1$}; \node at (9.4,1.6) {$1$};
\node at (9.4,1.2) {$0$}; \node at (9.6,1.4) {$0$}; \node at (9.8,1.2) {$0$};
\draw [->,thick] (9.85,1.55) to (10.35,2.05);
\draw [dotted, thick] (10,1.4) to (11.4,1.4);

\node at (9,-1.2) {$0$}; \node at (9.2,-1) {$1$}; \node at (9.4,-0.8) {$0$};
\node at (9.4,-1.2) {$1$}; \node at (9.6,-1) {$0$}; \node at (9.8,-1.2) {$0$};
\draw [->,thick] (9.75,-0.85) to (10.25,-0.35);
\draw [->,thick] (9.75,0.85) to (10.25,0.35);

\node at (10.2,0) {$0$}; \node at (10.4,0.2) {$1$}; \node at (10.6,0.4) {$0$};
\node at (10.6,0) {$0$}; \node at (10.8,0.2) {$0$}; \node at (11,0) {$0$};
\draw [->,thick] (10.95,0.35) to (11.45,0.85);

\node at (10.2,2.4) {$1$}; \node at (10.4,2.6) {$1$}; \node at (10.6,2.8) {$1$};
\node at (10.6,2.4) {$0$}; \node at (10.8,2.6) {$0$}; \node at (11,2.4) {$0$};
\draw [->,thick] (10.95,2.05) to (11.45,1.55);
\draw [dotted, thick] (11.2,0.2) to (12.4,0.2);

\node at (11.5,1.2) {$1$}; \node at (11.7,1.4) {$1$}; \node at (11.9,1.6) {$0$};
\node at (11.9,1.2) {$0$}; \node at (12.1,1.4) {$0$}; \node at (12.3,1.2) {$0$};
\draw [->,thick] (12.15,0.85) to (12.65,0.35);

\node at (12.6,0) {$1$}; \node at (12.8,0.2) {$0$}; \node at (13,0.4) {$0$};
\node at (13,0) {$0$}; \node at (13.2,0.2) {$0$}; \node at (13.4,0) {$0$};

\end{tikzpicture}

Bear in mind that, by \ref{3cats}, this is also a complete picture of the category of (finitely presented) functors on finitely presented $R$-modules, equivalently of the category of pp-sorts and pp-definable maps between them.  So everything about the, admittedly rather simple, model theory of $KA_3$-modules is contained in this picture.

This picture does illustrate much that is general, so let us look in detail at some of these $17$ indecomposable functors/pp-sorts.  Let us label them as $F_i$ with $i$ as in the following diagram.

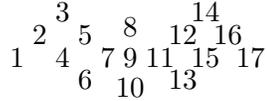
\begin{tikzpicture}
\node at (0,0) {$1$}; \node at (0.3,0.3) {$2$}; \node at (0.6,0.6) {$3$};
\node at (0.6,0) {$4$}; \node at (0.9,0.3) {$5$}; \node at (0.9,-0.3) {$6$};
\node at (1.2,0) {$7$};
\node at (1.5,0.4) {$8$}; \node at (1.5,0) {$9$}; \node at (1.5,-0.4) {$10$};
\node at (1.9,0) {$11$}; \node at (2.2,0.3) {$12$}; \node at (2.2,-0.3) {$13$};
\node at (2.5,0) {$15$}; \node at (2.5,0.6) {$14$};
\node at (2.8,0.3) {$16$}; \node at (3.1,0) {$17$};
\end{tikzpicture}

\noindent And let us also label the indecomposable $R$-modules as follows, locating each at its position in the Auslander-Reiten quiver of $R$.

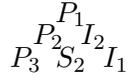
\begin{tikzpicture}
\node at (6.3,0) {$P_3$}; \node at (6.6,0.3) {$P_2$}; \node at (6.9,0.6) {$P_1$};
\node at (6.9,0) {$S_2$}; \node at (7.2,0.3) {$I_2$}; \node at (7.5,0) {$I_1$};
\end{tikzpicture}

\noindent Here $S_i$, $P_i$, $I_i$ respectively denote the simple module at vertex $i$, the projective cover of $S_i$ and the injective hull of $S_i$.  So we also have $S_3=P_3$, $I_3=P_1$, $S_1=I_1$.

Each indecomposable $R=KA_3$-module $N$ gives us the representable functor $(N,-)$ and these are precisely the indecomposable projective functors.  Consider the indecomposable functor $F_9$.  This is both projective and injective (one can see that from the AR-quiver since it neither ends nor begins an almost split sequence) and is readily identified as $(P_2,-)$ since, as is easily seen from the AR-quiver of $R$, the dimensions ${\rm dim}(P_2,X)$ as $X$ ranges over the six indecomposable $R$-modules, give exactly the pattern of $0$s and $1$s seen in $F_9$.  As a $3$-sorted module $P_2$ is cyclic, generated by an element of sort $2$, so $(P_2,-)$ is isomorphic to the pp-pair $(x_2=x_2)/(x_2=0)$.  If, instead, we think of $P_2$ as a 1-sorted $KA_3$-module, then we can use the pp-pair $(e_1x=0 \wedge e_3x=0)/(x=0)$.

Next, let us look at the factor $F_{11}$ of $F_9=(P_2,-)$:  clearly this is $(P_2,-)$ modulo its socle but it is perhaps easier to express this in terms of pp formulas for $R$-modules if we give $F_{11}$ as the cokernel of a map between representable functors.  So replace the socle of $(P_2,-)$ by its projective cover, which is clearly (from consulting the list of representable functors) $(I_2,-)$.  Therefore we get the exact sequence
$$
\begin{tikzpicture}
\node at (0,0) {$0$}; \node at (0.2,0.2) {$0$}; \node at (0.4,0.4) {$0$};
\node at (0.4,0) {$0$}; \node at (0.6,0.2) {$1$}; \node at (0.8,0) {$1$};
\end{tikzpicture}
\to
\begin{tikzpicture}
\node at (0,0) {$0$}; \node at (0.2,0.2) {$1$}; \node at (0.4,0.4) {$1$};
\node at (0.4,0) {$1$}; \node at (0.6,0.2) {$1$}; \node at (0.8,0) {$0$};
\end{tikzpicture}
\to
\begin{tikzpicture}
\node at (0,0) {$0$}; \node at (0.2,0.2) {$1$}; \node at (0.4,0.4) {$1$};
\node at (0.4,0) {$1$}; \node at (0.6,0.2) {$0$}; \node at (0.8,0) {$0$};
\end{tikzpicture}
\to 0
$$
That is,
$$(I_2,-) \xrightarrow{(f,-)} (P_2,-) \to F_{11} \to 0$$
where $f:P_2 \to I_2$ is the obvious morphism of $R$-modules.
Therefore, for any $R$-module $M$, $F_{11}(M) = (P_2,M)/{\rm im}(f,M)$ - the morphisms from $P_2$ to $M$ modulo those which factor through $f$, that is $e_2M/\alpha M$.  Therefore, as a pp-pair, we can write $F_{11}$ as $(x_2=x_2)/(\exists x_1 \,\, \alpha x_1 = x_2)$ if we are using the 3-sorted language for $R$-modules, or as $(x=e_2x)/(\exists y\,\, \alpha y=x)$ if using the 1-sorted language.

By \ref{3cats} every morphism in this category of pp-pairs is pp-definable.  As remarked already, in this example every non-isomorphism between indecomposables is a $K$-linear combination of compositions of the morphisms seen as arrows in the Auslander-Reiten quiver above (this follows from finite representation type using that the irreducible morphisms generate the radical of the ring with many objects which is $R\mbox{-}{\rm mod}$).  So, in order to give pp-definitions of the morphisms, it is enough to do that for those arrows.  That can be done explicitly and easily using the recipe from \cite{BurThes} for computing a pp formula which defines a given natural transformation between finitely presented functors (see the proof of \cite[10.2.30]{PreNBK}).  In fact, in this example, none of the maps is very ``interesting'', but we compute one nevertheless to illustrate the general procedure.

Consider the morphism (natural transformation) $\tau:F_{11} \to F_{16}$ which is got by composing three arrows/irreducible morphisms.  Regarding $F_{11}$ as a pp-pair $\phi/\psi$, and $F_{16}$ as $\phi'/\psi'$ (we could take $(e_3x=x)/(\exists y \,\, x=\beta\alpha y)$ for the latter), a free realisation\footnote{A {\bf free realisation} of a pp formula $\phi(\overline{x})$ is a finitely presented module $C$ and a tuple $\overline{c}$ from $C$ such that the pp-type of $\overline{c}$ in $C$ is generated by $\phi$.  This is equivalent to giving a surjection from a projective=representable functor onto $F_\phi$, see \cite[10.2.8]{PreNBK}.} of $\phi$ is $(P_2,c)$ where $c$ is any generator of $P_2$ - this is clear but also follows from \cite[10.2.25]{PreNBK}.  We look at the component of $\tau$ at $P_2$; the value of each of $F_{11}$ and $F_{16}$ at $P_2$ is $K$ and the component map is an isomorphism, $1_K$.  Therefore, an element $c'$ of $P_2$ such that the component $\tau_{P_2}$ takes $c+\psi(F_{11})$ to $c'+\psi'(F_{16})$ is just $c$ itself (in more interesting examples, this would be something different!).  A generator of the pp-type of $(c,c')$, that is $(c,c)$, is $x'=x$, so this is a pp formula which defines that natural transformation.

\vspace{4pt}

We will return to this example in Section \ref{secA3again} to illustrate localisation and some other things.

\section{Adding more conditions: localisation and definable subcategories} \label{secloc}

We may want to consider the model theory, not of all $R$-modules, but of some of them.
For instance, given a module $M$, it is model-theoretically natural to consider the class of modules elementarily equivalent to $M$.  It turns out to be very convenient to expand this to the class of all direct summands of modules elementarily equivalent to $M$.  Then this is a typical definable subcategory of $R\mbox{-}{\rm Mod}$ in the sense of the following definition (I am now using $R$ to denote any skeletally small preadditive category).

We say that a subcategory of $R\mbox{-}{\rm Mod}$ (full, and closed under isomorphisms) is {\bf definable} if it is closed under direct products, direct limits and pure submodules.  Recall that an embedding $f:A\to B$ between $R$-modules is {\bf pure} if, for every pp formula $\phi$, we have $\phi(A) = A^n\cap \phi(B)$, where $n$ is the number of free variables of $\phi$.

\begin{theorem} (see, e.g., \cite[\S 3.4]{PreNBK}) The following are equivalent for a subcategory ${\mathcal D}$ of $R\mbox{-}{\rm Mod}$ (full and closed under isomorphisms):

\noindent (i) ${\mathcal D}$ is a definable subcategory of $R\mbox{-}{\rm Mod}$;

\noindent (ii) there is a module $M$ such that ${\mathcal D}$ is the class of direct summands of modules elementarily equivalent to direct sums\footnote{If, for instance, $R$ is an algebra over an infinite field, then it is enough to take modules which are direct summands of modules elementarily equivalent to $M$.} of copies of $M$;

\noindent (iii) there is a set $\Phi$ of pp-pairs for $R$-modules such that ${\mathcal D} = \{ M\in R\mbox{-}{\rm Mod}: \phi(M)/\psi(M)=0 \,\, \forall \phi/\psi\in \Phi\}$.
\end{theorem}

For pp-pairs, see Section \ref{seceq+}.

It is the case that every definable subcategory ${\mathcal D}$ is closed under pure-injective hulls and has the property that if $0\to A \to B \to C \to 0$ is a {\bf pure-exact sequence} (that is, if the sequence is exact and $A\to B$ is a pure embedding), then $B\in {\mathcal D}$ iff $A, C \in {\mathcal D}$.

We say that a module $M$ {\bf generates} the definable category ${\mathcal D}$ if ${\mathcal D}$ is the smallest definable category containing it (then $M$ will be as in (ii)).  A generator $N$ such that the modules in ${\mathcal D}$ are those which are pure in direct products of copies of $N$ is said to be an {\bf elementary cogenerator} for ${\mathcal D}$ (such a module exists, given ${\mathcal D}$, see \cite[5.3.52]{PreNBK}).

Thus the definable subcategories of $R\mbox{-}{\rm Mod}$ are obtained by declaring certain pp-sorts to be $0$.  Note that closure of such a pp-pair $\phi/\psi$ is expressed by the sentence $\forall \overline{x} \, (\phi(\overline{x}) \rightarrow \psi(\overline{x}))$ (the other implication holds by our definition of pp-pair).

Given a definable subcategory ${\mathcal D}$ of $R\mbox{-}{\rm Mod}$, we set ${\mathcal S}_{\mathcal D} = \{ \phi/\psi: \phi(D)/\psi(D)=0 \,\, \forall \, D\in {\mathcal D}\}$ - the set of pp-pairs closed on every module in ${\mathcal D}$.  This is a {\bf Serre subcategory} of the category $_R{\mathbb L}^{\rm eq+}$ of pp-pairs for $R$-modules, meaning that if $0 \to A \to B \to C \to 0$ is an exact sequence in the category of pp-pairs, then $B\in {\mathcal S}_{\mathcal D}$ iff $A, C \in {\mathcal S}_{\mathcal D}$.  If we prefer to view that category as the category of finitely presented functors on finitely presented modules (\ref{3cats}) then the definition becomes ${\mathcal S}_{\mathcal D} = \{ F\in (R\mbox{-}{\rm mod}, {\bf Ab})^{\rm fp}: \overrightarrow{F}D=0 \,\, \forall \, D\in {\mathcal D}\}$.

\begin{theorem}  Given a ring, or skeletally small preadditive category, $R$, there is a natural bijection ${\mathcal D} \leftrightarrow {\mathcal S}_{\mathcal D}$ between definable subcategories of $R\mbox{-}{\rm Mod}$ and Serre subcategories of the free abelian category ${\rm Ab}(R)$ on $R$.
\end{theorem}

These also are in natural bijection with the closed subsets of the Ziegler spectrum of $R$:  a topological space which has the isomorphism classes of indecomposable pure-injective $R$-modules for its points and the supports of pp-pairs = finitely presented functors for a basis of open sets.  See, for example, \cite{Zie}, \cite{PreNBK} for details (which we don't need here since our focus now will be on the functor category).

\vspace{4pt}

If ${\mathcal D}$ is a definable subcategory of $R\mbox{-}{\rm Mod}$ then, of course, we may do model theory for the modules in ${\mathcal D}$ using a language for $R$-modules.  But there is a more intrinsic language associated to ${\mathcal D}$, obtained by localising the category of pp-pairs for $R$-modules at the Serre subcategory ${\mathcal S}_{\mathcal D}$.  This uses the following general construction.

\begin{theorem}  Suppose that ${\mathcal A}$ is an abelian category and let ${\mathcal S}$ be a Serre subcategory of ${\mathcal A}$.  Then there is an {\bf quotient} abelian category ${\mathcal A}/{\mathcal S}$ and an exact functor $Q:{\mathcal A} \to {\mathcal A}/{\mathcal S}$ such that $Q{\mathcal S} =0$ and with the property that if ${\mathcal B}$ is an abelian category and $G:{\mathcal A} \to {\mathcal B}$ is an exact functor with $G{\mathcal S}=0$, then there is an essentially unique factorisation of $G$ through $Q$ as $G=HQ$, and $H$ is exact.

$\xymatrix{
{\mathcal A} \ar[r]^Q \ar[dr]_{G} & {\mathcal A}/{\mathcal S} \ar[d]^{H \mbox{ exact}} \\ & {\mathcal B}
}$
\end{theorem}

If we let ${\mathcal A}$ be the free abelian category on / the category of pp-pairs for $R$ and then take $G$ above to be ${\rm ev}_D$ - evaluation-at-$D$ - for any module $D\in {\mathcal D}$, we see that this factors through the quotient ${\rm Ab}(R) \to {\rm Ab}(R)/{\mathcal S}_{\mathcal D}$.  We denote this quotient category ${\rm Ab}(R)/{\mathcal S}_{\mathcal D}$ by ${\mathcal A}({\mathcal D})$ and refer to it as the {\bf functor category} of ${\mathcal D}$.  It is indeed the category of pp-pairs for ${\mathcal D}$ - the localisation process may be regarded as simply restricting functors/pp-pairs from $R\mbox{-}{\rm Mod}$ to ${\mathcal D}$ (and of course certain of those become $0$ on restricting to ${\mathcal D}$, certain non-isomorphic functors become isomorphic when restricted to ${\mathcal D}$, certain definable relations become functional, or total, or even invertible - so there will in general be more maps though, up to isomorphism, fewer objects; we will illustrate all this in the next section).  This abelian category ${\mathcal A}({\mathcal D})$ of functors is intrinsically associated to ${\mathcal D}$ - note that ${\mathcal D}$ may be found, up to equivalence, as a definable subcategory of many module categories.  We say that an additive category ${\mathcal D}$ is {\bf definable} if it is equivalent to a definable subcategory of some (possibly multisorted) module category.

\begin{theorem} \label{funD} (see \cite[12.10, 10.8]{PreMAMS}) Suppose that ${\mathcal D}$ is a definable category.  Then its category ${\mathcal A}({\mathcal D})$ of pp-pairs is equivalent to the category $({\mathcal D}, {\bf Ab})^{\prod \rightarrow}$ of additive functors from ${\mathcal A}$ to ${\bf Ab}$ which commute with direct products and direct limits.

On the other hand, evaluation-at-$D$ for $D\in {\mathcal D}$, gives an equivalence of ${\mathcal D}$ with the category of exact functors on ${\mathcal A}({\mathcal D})$:  ${\mathcal D} \simeq {\rm Ex}({\mathcal D}, {\bf Ab})$.
\end{theorem}

If ${\mathcal D}$ is the definable subcategory of $R\mbox{-}{\rm Mod}$ generated by a module $M$, then we also write ${\mathcal A}(M)$ for ${\mathcal A}({\mathcal D})$.  Note that $M^{\rm eq+}$, an exact functor on ${\rm Ab}(R)$ factors through ${\mathcal A}(M)$, indeed if we think of $M^{\rm eq+}$ as an exact functor on ${\mathcal A}(M)$ then it is faithful and is a more intrinsic representation of $M$ than was the functor on ${\rm Ab}(R)$.  In the paper \cite{PreExact} I argue that the functor $M^{\rm eq+}:{\mathcal A}(M) \to {\bf Ab}$ is the most complete and intrinsic representation of any module; in particular this representation includes all the ways that a single module may be found (literally or up to inter-interpretability) over many different rings.

\section{An example: $KA_3$ again} \label{secA3again}

We will compute some localisations of the functor category for $KA_3$-modules, $K$ a field, corresponding to some definable subcategories of $KA_3\mbox{-}{\rm Mod}$.  The following result tells us that each will be the category of finitely presented modules over the endomorphism ring of a suitable generator of the definable category.  The result extends the appearance of the Auslander algebra (in the case where the definable category is a category of modules) and it appears (e.g.\cite[\S 1.2]{Nori}) in this rather more general form in Nori's work on motives.  It may well appear elsewhere.  I include a proof.

\begin{theorem} \label{functendorng} Suppose that $M$ is a finitely presented (right or left) $R$-module which is noetherian over its endomorphism ring $S={\rm End}(M)$, which we let act on the left.  Then ${\mathcal A}(M)$ - the functor category of the definable subcategory generated by $M$ =  the category of pp-pairs for $M$, is equivalent to $S\mbox{-}{\rm mod}$.

The equivalence is given, in one direction, by taking a functor $F$ to $FM$.
\end{theorem}
\begin{proof}  First note that the functor ${\rm ev}_M$ which takes each $F\in {\mathcal A}(M)$ to its value at $M$, is an equivalence of ${\mathcal A}(M)$ with its image (a non-full subcategory of ${\bf Ab}$).  We will prove the equivalence of $S\mbox{-}{\rm mod}$ and this category, which is the category of pp-definable subgroups of $M$ and pp-definable maps between them.  We know that every subgroup of $M^n$ pp-definable in $M$ is an ${\rm End}(M)=S$-submodule of $M^n$ and that every pp-definable map commutes with that $S$-module structure, so we have to show that the module $_SS$, with its full endomorphism ring, occurs in this way.  That will be enough since the smallest abelian subcategory of $S\mbox{-}{\rm mod}$ containing $_SS$ with its full endomorphism ring is $S\mbox{-}{\rm mod}$ itself.

Choose a finite generating tuple $\overline{a} = (a_1,\dots,a_n)$ for $M$ as an $R$-module.  Then $S \to M^n$ defined by $s\in S \mapsto s(\overline{a}) = (sa_1, \dots, sa_n)$ embeds $S$ as a left $S$-submodule of $M^n$.  Since, by the finiteness conditions on $M$, every $S$-submodule of $M^n$ is pp-definable (see \cite[1.2.12]{PreNBK}), it follows that there is a pp formula $\phi(\overline{x})$ such that $\phi(M) =S\overline{a}$.  We must show that the right multiplications by elements of $S$ (the endomorphisms of $_SS$) and the images under ${\rm ev}_M$ of the pp-definable endomorphisms of the sort $\phi(\overline{x})/(\overline{x}= \overline{0})$ coincide.

So let $t\in S$ and consider a pp formula $\rho_t(\overline{x},\overline{x}')$ which generates the pp-type in $M$ of $(\overline{a},t\overline{a})$ (this exists since $M$ is a finitely presented $R$-module, see \cite[1.2.6]{PreNBK}).  So $M\models \rho_t(\overline{a},t\overline{a})$.  Then, for every $s\in S$, we have $M\models \rho_t(s\overline{a},st\overline{a})$.  Since every $\overline{b}\in \phi(M)$ has the form $s\overline{a}$ for some $s\in S$, the relation defined by $\rho_t$ is therefore total on $\phi(M)$. It is also functional since if $s\in S$ is such that $s\overline{a}=\overline{0}$ then $sM=0$ so also $st\overline{a} = (sta_1, \dots, sta_n) =\overline{0}$.  Therefore right multiplication by $t$ is a pp-definable map, as required.  (Note that our identification of $_SS$ with $\phi(M)$ pairs $s\in S$ with $s\overline{a}$, so $\rho_t$ is acting as right multiplication by $t$.)

Suppose, for the converse, that $\rho$ is a pp formula defining an endomorphism of $\phi(M)$.  Say $M\models \rho(\overline{a},\overline{b})$.  Since $\overline{b}\in \phi(M)=S\overline{a}$, there is $t\in S$ such that $ta_i=b_i$ for each $i$ ($t$ is unique, since $\overline{a}$ generates $M_R$).  For each $s\in S$ we have $M\models \rho(s\overline{a},s\overline{b}=st\overline{a})$, so the action of $\rho$ on $\phi(M)$ is right multiplication by $t$, as required.
\end{proof}

In the other direction in \ref{functendorng}, the equivalence takes a finitely presented $S$-module $A$ to the functor determined by taking $M$ to $A$.  We can be more explicit.  Take a projective presentation of $A$:  $S^l \to S^m \to A \to 0$.  Consider the exact sequence of pp-sorts which is $\phi^l \xrightarrow{\eta} \phi^m \to {\rm coker}(\eta)$ where $\phi$ is as in the proof above and where the pp formula $\eta$ defines the morphism between projectives which appears in the presentation of $A$ (and where $\phi^m$ means the direct sum of $m$ copies of the pp-pair $\phi(\overline{x})/(\overline{x} = \overline{0})$).  Then the functor corresponding to $A$ is given by any pp-pair which defines the functor ${\rm coker}(\eta)$, for example $\big(\bigwedge_{j=1}^m \, \phi(\overline{x^j})\big)/\big( \exists \, \overline{y^1}, \dots, \overline{y^l} \,\, \eta(\overline{y^1}, \dots, \overline{y^l}, \overline{x^1}, \dots, \overline{x^m}) \big)$.

The case we are interested in - definable subcategories generated by finitely many finite-length modules over an artin algebra - will also be covered by the following variation.

\begin{cor} \label{functendorng2} Let ${\mathcal D}$ be a definable subcategory of $R\mbox{-}{\rm Mod}$ which consists of the direct sums of copies of the finitely many $R$-modules $N_1, \dots, N_t$.

Then ${\mathcal A}({\mathcal D}) = {\rm End}(\bigoplus_{i=1}^t N_i) \mbox{-}{\rm mod}$.

If $A$ is a finitely presented module over the endomorphism ring $S={\rm End}(\bigoplus_{i=1}^t N_i)$, then the functor it defines on ${\mathcal D}$ is given on objects by $D \mapsto {\rm Hom}_S({\rm Hom}_R(D, \bigoplus_{i=1}^t N_i), A)$.
\end{cor}
\begin{proof}  We follow the account of the Auslander algebra in \cite[\S 4.9]{BenBk1}).

Set $M = \bigoplus_{i=1}^t N_i$.  The functor from ${\mathcal A}({\mathcal D})$ to $S\mbox{-}{\rm Mod}$ is defined by taking $F\in{\mathcal A}(D)$ to $FM$, which, as a subgroup pp-definable in $M$, is an $S$-submodule (of some power of $M$) and taking a natural transformation, which we can regard as a pp-definable map, $\rho:F\to G$ in ${\mathcal A}({\mathcal D})$, to $\rho(M):F(M) \to G(M)$.  As a pp-definable action, that commutes with the action of $S$.  So this is just as in \ref{functendorng}.

In the other direction, we associate a functor $F_A$ on ${\mathcal D}$ to each finitely presented left $S$-module $A$ according to the formula given.  We must check that these are inverse processes.

The functor $F_A$ applied to $M$ gives ${\rm Hom}_S({\rm Hom}_R(M,M), A) = {\rm Hom}_S(S, A) = A$, as required.  In the other direction, given a functor $G$, we have $F_{GM}: D \to  {\rm Hom}_S({\rm Hom}_R(D,M), GM)$ so $F_{GM}(M)= GM$ and, given that every $D\in {\mathcal D}$ is a direct sum of copies of direct summands of $M$, two functors which agree on $M$ are the same functor, so $F_{GM}=G$, as required (checking the actions on morphisms is left as an exercise).
\end{proof}

\begin{example}  We consider the $KA_3$-module $T=P_1\oplus P_2 \oplus S_2$ and set $R'={\rm End}_R(T)$ to be its endomorphism ring, acting on the left.  Thus we can regard $T$ as both an $R$-module and an $R'$-module.  It is the case that these actions are inter-definable.  That is, the structure of $T$ as an $R'$-module is already part of $(_RT)^{\rm eq+}$ and the structure of $T$ as an $R$-module is already part of $(_{R'}T)^{\rm eq+}$.  Essentially, the same module $T$ lies over (at least) two different rings.  As an enriched, $^{\rm eq+}$-structure it is just one structure and the choice of which ring to view it over is simply a choice of which (finite) generating set of sorts to extract as ``the ring'' (literally, the endomorphism ring in the functor category of the direct sum of those sorts is ``the ring'').  See \cite{PreExact} for a development of this theme.

Those familiar with tilting theory will recognise that $T$ is a tilting module and that what we are illustrating here is the fact that classical tilting can be seen as making a new (but reversible) choice of generating sorts of a module.

The functor $(T,-)$ gives an equivalence between the definable subcategory of $R$-modules which satisfy ${\rm Ext}^1_R(T,-)=0$ and the definable subcategory of $R'$-modules which satisfy ${\rm Tor}_1^{R'}(_{R'}T,-)=0$.  (If a module $A$ is ${\rm FP}_2$ then both functors ${\rm Ext}^1(A,-)$ and ${\rm Tor}_1(A,-)$ are given by pp-pairs, see \cite[\S 10.2.6]{PreNBK}.)

The definable subcategory of $R\mbox{-}{\rm Mod}$ defined by ${\rm Ext}^1_R(T,-)=0$ coincides, since $T$ is tilting, with that generated by $T$, hence it is the definable subcategory ${\mathcal D}$ of those $R$-modules which do not contain $P_3$ as a direct summand, that is, of modules which are direct sums of copies of the other five indecomposable $R$-modules.  To compute the corresponding localisation ${\mathcal A}({\mathcal D}) ={\rm Ab}(R)/{\mathcal S}_{\mathcal D}$ of the functor category, we should determine the category ${\mathcal S}_{\mathcal D}$ of functors which are $0$ on ${\mathcal D}$.  By inspection of the AR-quiver in Section \ref{secA3}, there is just one such functor, namely $F_{17}$, so ${\mathcal S}_{\mathcal D}$ consists of the finite direct sums of copies of $F_{17}$.  The effect of localisation is to make this functor $0$ and also, any morphism with kernel or cokernel in ${\mathcal S}_{\mathcal D}$ becomes an isomorphism in the quotient category.  So we see that the irreducible morphisms from $F_{12}$ to $F_{14}$ and from $F_{15}$ to $F_{16}$ become isomorphisms.  No other irreducible morphisms are affected by this localisation so, since every morphism is a linear combination of compositions of irreducible maps, the quotient category ${\mathcal A}({\mathcal D})$ is as follows, where we have represented functors by showing the dimensions of their values on the five indecomposable modules in ${\mathcal D}$.

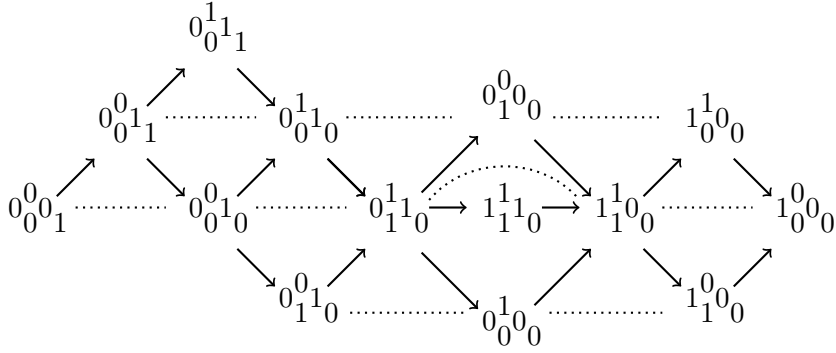
\begin{tikzpicture}

\node at (0.2,0.2) {$0$}; \node at (0.4,0.4) {$0$};
\node at (0.4,0) {$0$}; \node at (0.6,0.2) {$0$}; \node at (0.8,0) {$1$};
\draw [->,thick] (0.75,0.35) to (1.25,0.85);
\draw [dotted, thick] (1,0.2) to (2.2,0.2);

\node at (1.4,1.4) {$0$}; \node at (1.6,1.6) {$0$};
\node at (1.6,1.2) {$0$}; \node at (1.8,1.4) {$1$}; \node at (2,1.2) {$1$};
\draw [->,thick] (1.95,1.55) to (2.45,2.05);
\draw [->,thick] (1.95,0.85) to (2.45,0.35);
\draw [dotted, thick] (2.2,1.4) to (3.4,1.4);

\node at (2.6,2.6) {$0$}; \node at (2.8,2.8) {$1$};
\node at (2.8,2.4) {$0$}; \node at (3,2.6) {$1$}; \node at (3.2,2.4) {$1$};
\draw [->,thick] (3.15,2.05) to (3.65,1.55);

\node at (2.6,0.2) {$0$}; \node at (2.8,0.4) {$0$};
\node at (2.8,0) {$0$}; \node at (3,0.2) {$1$}; \node at (3.2,0) {$0$};
\draw [->,thick] (3.15,0.35) to (3.65,0.85);
\draw [->,thick] (3.15,-0.35) to (3.65,-0.85);
\draw [dotted, thick] (3.4,0.2) to (4.6,0.2);

\node at (3.8,1.4) {$0$}; \node at (4,1.6) {$1$};
\node at (4,1.2) {$0$}; \node at (4.2,1.4) {$1$}; \node at (4.4,1.2) {$0$};
\draw [->,thick] (4.35,0.85) to (4.85,0.35);
\draw [dotted, thick] (4.6,1.4) to (6,1.4);

\node at (3.8,-1) {$0$}; \node at (4,-0.8) {$0$};
\node at (4,-1.2) {$1$}; \node at (4.2,-1) {$1$}; \node at (4.4,-1.2) {$0$};
\draw [->,thick] (4.35,-0.85) to (4.85,-0.35);
\draw [dotted,thick] (4.65,-1.2) to (6.2,-1.2);

\node at (5,0.2) {$0$}; \node at (5.2,0.4) {$1$};
\node at (5.2,0) {$1$}; \node at (5.4,0.2) {$1$}; \node at (5.6,0) {$0$};
\draw [->,thick] (5.6,0.4) to (6.35,1.15);
\draw [->,thick] (5.6,-0.4) to (6.35,-1.15);
\draw [->,thick] (5.7,0.2) to (6.2,0.2);
\draw [dotted,thick] (5.7,0.3) to [out=50, in=130] (7.7,0.3);

\node at (6.5,1.7) {$0$}; \node at (6.7,1.9) {$0$};
\node at (6.7,1.5) {$1$}; \node at (6.9,1.7) {$0$}; \node at (7.1,1.5) {$0$};
\draw [->,thick] (7.2,0.2) to (7.7,0.2);
\draw [->,thick] (4.35,0.85) to (4.85,0.35);
\draw [dotted,thick] (7.35,1.4) to (8.8,1.4);

\node at (6.5,0.2) {$1$}; \node at (6.7,0.4) {$1$};
\node at (6.7,0) {$1$}; \node at (6.9,0.2) {$1$}; \node at (7.1,0) {$0$};
\draw [->,thick] (7.1,1.1) to (7.85,0.35);

\node at (6.5,-1.3) {$0$}; \node at (6.7,-1.1) {$1$};
\node at (6.7,-1.5) {$0$}; \node at (6.9,-1.3) {$0$}; \node at (7.1,-1.5) {$0$};
\draw [->,thick] (7.1,-1.1) to (7.85,-0.35);
\draw [dotted, thick] (7.3,-1.2) to (8.8,-1.2);

\node at (8,0.2) {$1$}; \node at (8.2,0.4) {$1$};
\node at (8.2,0) {$1$}; \node at (8.4,0.2) {$0$}; \node at (8.6,0) {$0$};
\draw [->,thick] (8.55,0.35) to (9.05,0.85);
\draw [->,thick] (8.55,-0.35) to (9.05,-0.85);
\draw [dotted, thick] (8.8,0.2) to (10,0.2);

\node at (9.2,1.4) {$1$}; \node at (9.4,1.6) {$1$};
\node at (9.4,1.2) {$0$}; \node at (9.6,1.4) {$0$}; \node at (9.8,1.2) {$0$};
\draw [->,thick] (9.75,0.85) to (10.25,0.35);

\node at (9.2,-1) {$1$}; \node at (9.4,-0.8) {$0$};
\node at (9.4,-1.2) {$1$}; \node at (9.6,-1) {$0$}; \node at (9.8,-1.2) {$0$};
\draw [->,thick] (9.75,-0.85) to (10.25,-0.35);

\node at (10.4,0.2) {$1$}; \node at (10.6,0.4) {$0$};
\node at (10.6,0) {$0$}; \node at (10.8,0.2) {$0$}; \node at (11,0) {$0$};

\end{tikzpicture}

This, then, is the category of pp-pairs for the definable category ${\mathcal D}$ consisting of modules $M$ such that ${\rm Ext}^1_R(T,M)=0$.  As follows from \ref{functendorng2}, it is the category of modules over the endomorphism ring of the direct sum of the five indecomposables in ${\mathcal D}$.  Indeed that ring is just the path algebra of the quiver with relations that we obtain by deleting the vertex $P_3$ from the Auslander-Reiten quiver of $A_3$ (in Section \ref{secmodrepfun}).

The definable category ${\mathcal D}$ strictly contains the definable subcategory ${\mathcal E}$ generated by $T$.  That category consists of direct sums of copies of the indecomposable direct summands $P_1$, $P_2$, $S_2$ of $T$.  The category ${\mathcal A}({\mathcal E})$ of pp-pairs for ${\mathcal E}$ can be obtained as a localisation of ${\rm Ab}(R)$ or, since it is a definable subcategory of ${\mathcal D}$, as a localisation of the category ${\mathcal A}({\mathcal D})$, which we have just computed.  Again, one can compute using the AR-quivers of these categories; note that, in this case, various indecomposables become isomorphic when localised and also the indecomposable
\begin{tikzpicture}
\node at (5,0.2) {$0$}; \node at (5.2,0.4) {$1$};
\node at (5.2,0) {$1$}; \node at (5.4,0.2) {$1$}; \node at (5.6,0) {$0$};
\end{tikzpicture}
becomes decomposable when localised (localisation functors are exact but, outside the classical rings of fractions context, in general far from full!).  One easily computes that the category ${\mathcal A}({\mathcal E})$ has the following AR-quiver.

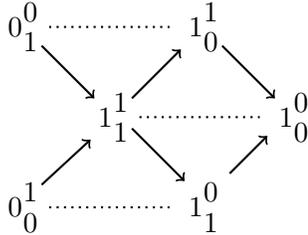
\begin{tikzpicture}
\node at (3.4,2.4) {$1$}; \node at (3.2,2.6) {$0$}; \node at (3.4,2.8) {$0$};
\node at (3.4,0) {$0$}; \node at (3.2,0.2) {$0$}; \node at (3.4,0.4) {$1$};
\node at (4.6,1.2) {$1$}; \node at (4.4,1.4) {$1$}; \node at (4.6,1.6) {$1$};
\node at (5.8,2.4) {$0$}; \node at (5.6,2.6) {$1$}; \node at (5.8,2.8) {$1$};
\node at (5.8,0) {$1$}; \node at (5.6,0.2) {$1$}; \node at (5.8,0.4) {$0$};
\node at (7,1.2) {$0$}; \node at (6.8,1.4) {$1$}; \node at (7,1.6) {$0$};

\draw [->,thick] (3.55,2.35) to (4.25,1.65);
\draw [->,thick] (3.55,0.5) to (4.2,1.1);
\draw [->,thick] (4.75,1.65) to (5.4,2.3);
\draw [->,thick] (4.75,1.25) to (5.45,0.55);
\draw [->,thick] (5.95,2.35) to (6.65,1.65);
\draw [->,thick] (6.05,0.65) to (6.55,1.15);

\draw [dotted, thick] (3.65,0.2) to (5.25,0.2);
\draw [dotted, thick] (3.65,2.6) to (5.25,2.6);
\draw [dotted, thick] (4.85,1.4) to (6.45,1.4);
\end{tikzpicture}

This is the path algebra of the quiver which is the subquiver of the AR-quiver of $A_3$ on the vertices $P_2$, $P_3$ and $S_2$ - that subquiver is $A_3$ but with a different orientation, the middle vertex pointing to each of the other two.  As should be, since the effect of this kind of tilting is to change orientation of arrows in a quiver.  In accordance with \ref{functendorng}, the quiver, above, of the localised functor category is that of the module category over the endomorphism ring $R'$ of $T$, and $R'$ is the path algebra of $A_3$ with this new orientation.
\end{example}

\section{Further examples:  triangulated categories}

Compactly generated triangulated categories are somewhat analogous to finitely accessible categories.  Triangulated categories are not abelian, nor do they have direct limits, so many of the standard constructions of model theory, are not available.  Nevertheless, a model theory may be set up.  As with accessible categories, if finitary model theory is to be used (that is model theory using formulas which are finite, not infinite, strings of symbols) then we need some ``finitary" objects to generate the category in some sense.  In the triangulated context ``compact'' objects replace the finitely presented ones, where we say that an object $C$ of the triangulated category ${\mathcal T}$ is {\bf compact} if the functor $(C,-)$ commutes with infinite direct sums (we will assume that ${\mathcal T}$ has infinite coproducts - a fairly mild assumption).  Then the triangulated category ${\mathcal T}$ is {\bf compactly generated} if it has infinite direct sums, if there is, up to isomorphism, a set of compact objectss and if, for every non-zero object $T\in {\mathcal T}$, there is a non-zero morphism from a compact object to $T$.  That is, the subcategory ${\mathcal T}^{\rm c}$ of compact objects generates ${\mathcal T}$ in the sense that it sees every object (but usually there will be morphisms not seen by the compact objects).

One can then set up a multisorted language based on (a small version of) the full subcategory ${\mathcal T}^{\rm c}$ of compact objects and regard every object $T\in {\mathcal T}$ as a structure for this language by taking the objects of $T$ of sort $C\in {\mathcal T}^{\rm c}$ to be the morphisms in $(C,T)$, just as for finitely accessible categories (Example \ref{exfinacc}).  Some first observations about this language, including a quantifier-elimination result is given in \cite[\S 3]{GarkPre2} (see also \cite{GarkPre1}).

Notice that this language essentially treats an object $T\in {\mathcal T}$ as the restricted representable functor $(-,T):({\mathcal T}^{\rm c})^{\rm op} \to {\bf Ab}$, that is, as a right ${\mathcal T}^{\rm c}$-module.  So it can be seen as moving ${\mathcal T}$ into the context of a (multisorted) module category.  The restricted Yoneda functor, which does the moving, is neither faithful (there are morphisms not seen by the compact objects) nor full, though it is full on pure-injectives (\cite[1.7]{KraTel}).  But, as is shown in the work of Krause and Beligiannis (\cite{KraTel}, \cite{Belig}, \cite{BeligPur}), also \cite{BenGnac}, in developing purity in triangulated categories, this is a very useful embedding nevertheless.  In practice it gives a model theory for triangulated categories which can be used in computations in particular examples and in proving general results - see \cite{ALPP}.

\section{Further examples:  Nori motives}

The very rough idea (of Grothendieck) is that the motive of a variety is its abelian avatar:  given a suitable category ${\mathcal V}$ of varieties (or schemes), there should be a functor from ${\mathcal V}$ to its category of motives.  That category should be abelian and such that every homology or cohomology theory on ${\mathcal V}$ factors through the functor from ${\mathcal V}$ to its category of motives.  So that functor itself should be a kind of universal (co)homology theory for ${\mathcal V}$.  There are a number of accounts of the general idea, for example \cite{BVMot}, \cite{Milne}.  Throughout we assume that the base field $K$ is a subfield of the complex numbers.

In the case that ${\mathcal V}$ is the category of nonsingular projective varieties over ${\mathcal C}$, there is such a category of motives.  An account of its construction can be found in, for example, \cite{Kleiman}.  But the question of existence for possibly singular, not-necessarily projective varieties - the conjectural category ${\mathcal M}{\mathcal M}$ of mixed motives - is open.

In the 90s Nori described the construction of an abelian category which is a candidate for the category of mixed motives, \cite{Nori}.  His idea is to construct from a category of varieties ${\mathcal V}$ a (very large) quiver $D$ such every (co)homology theory on ${\mathcal V}$ gives a representation of $D$ (or $D^{\rm op}$).  A particular representation is used to construct this category - namely singular homology.  There is more involved than this, in particular a product structure on $D$ is needed to give a tensor product operation on the category of motives.

In fact, it turns out that Nori's category of motives is exactly the abelian category of functors/pp-pairs associated to the representation given by singular homology.  In essence this first appeared in \cite{BVCL}, though it is not said this way.  There Caramello used the methods of categorical model theory, in particular classifying toposes for regular logic, and showed that Nori's category is the effectivisation of the regular syntactic category for the regular theory associated to Nori's diagram $D$.  This is a much simpler construction than Nori's original one, in particular there is no need to approximate the final result through finite subdiagrams of $D$ or to go {\it via} coalgebra representations.

In \cite{BVCL} additivity appears at a relatively late stage of the construction.  If we build that in from the beginning then, as we show in \cite{BVP}, we are able to apply the existing model theory of additive structures and, in particular, realise Nori's category of motives as a category of pp-pairs (equivalently as a localisation of the free abelian category on the preadditive category ${\mathbb Z}\overrightarrow{D}$ generated by Nori's diagram $D$).

In brief, Nori's diagram is as follows. For a detailed recent account, which also describes the relations to questions about period numbers, see \cite{HMS}.

For the vertices, we take triples $(X,Y,i)$ where $X,Y\in {\mathcal V}$, $Y$ is a closed subvariety of $X$ and $i\in {\mathbb Z},$.  The arrows of $D$ are of two kinds.  For each morphism $f:X\to X'$ of ${\mathcal V}$ we have a corresponding arrow $(X,Y,i) \to (X',Y',i)$ provided $fY\subseteq Y'$, for each $i$.  Further, for each $X,Y,Z \in {\mathcal V}$ with $Y \supseteq Z$ closed subvarieties of $X$, we add an arrow $(Y,Z,i) \to (X,Y,i-1)$.  A homology theory $H$ on ${\mathcal V}$ then gives a representation of this quiver by sending $(X,Y,i)$ to the relative homology $H_i(X,Y)$.  Arrows of the first kind are sent to the obvious maps between relative homology objects; those of the second kind are send to the connecting maps in the long exact sequence for homology.  Taking $H$ to be singular homology, we obtain a representation of $D$ and then the category ${\mathcal A}({\mathcal D})$ of pp-pairs for this representation turns out to be Nori's category of motives.

As mentioned above, an important additional feature there is that there should be a tensor product structure on motives.  This is needed, for example, to express the K\"{u}nneth formula.  In recent work \cite{BVHP} with Barbieri-Viale and Huber we extend the construction of \cite{BVP} to include an induced tensor product.  In particular we show how a tensor product on the category of $R$-modules induces a tensor product on the free abelian category ${\rm Ab}(R)$, that is, we show how to define the induced tensor operation on pp-pairs; see the following section.

\section{Extending tensor product to sorts; an example} \label{sectensor}

In the previous section, we saw how Nori motives may be obtained as the category of pp-sorts for a multisorted module.  As mentioned there, an important additional feature is that there should be a tensor product structure on motives.  Here I briefly report on work  \cite{BVHP} with Barbieri-Viale and Huber which extends the constructions of \cite{BVCL} and \cite{BVP} to an induced tensor product and I present an example, in considerable detail, illustrating a key part of that process, namely the extension of a tensor product on the category of modules to a tensor product on the category of pp-pairs.  For the general construction see \cite{BVHP}.

\vspace{6pt}

Any commutative ring $R$ has the usual tensor product over $R$ on $R\mbox{-}{\rm Mod}$.  We choose a very simple example, namely $R=K[\epsilon:\epsilon^2=0]$.  An alternative would be to take $R={\mathbb Z}_4$:  then all the computations would be essentially the same, see \cite[\S 6.6]{Perera}.  We compute the category of pp-pairs, that is, the free abelian category ${\rm Ab}(R)$, and the induced tensor product on it.  The computation of the category is, in fact, done elsewhere, \cite[\S 6.8]{Perera}, \cite[4.3]{PreAxtFlat} but, rather than computing ${\rm Ab}(R)$ as the category of modules over the Auslander algebra of $R$ (as in those references), here we will compute it as the category $(R\mbox{-}{\rm mod}, {\bf Ab})^{\rm fp}$ of finitely presented functors on finitely presented modules.  It seems very unclear how to compute it directly as a category of pairs of pp formulas, hence the usefulness of having these alternative, algebraic, descriptions of such categories.

\subsection{Computing ${\rm Ab}(K[\epsilon])$}

Every $R$-module is a direct sum of copies of the two indecomposable modules, $R$ and the unique simple module $K$ which is $R$ modulo its radical and is also isomorphic to the socle of $R$.

There is essentially one interesting exact sequence in ${\rm Mod}\mbox{-}R$, namely $$0 \to K \xrightarrow{j} R \xrightarrow{p} K \to 0,$$ where $j$ takes a chosen generator of $K$ to $\epsilon\in R$ and $p$ is projection taking $1\in R$ to $\epsilon$.  Note that multiplication by $\epsilon$ on $R$ is the composition $jp$.

Each indecomposable $R$-module gives an indecomposable representable functor: ($\ast$)

\noindent $(R,-)$ is the forgetful functor (identify functor if we think of it taking values in $R\mbox{-}{\rm mod}$ rather than the category of $K$-vector spaces);

\noindent $(K,-)$ satisfies $(K,K)=K$, $(K,R)=K$,

\noindent $(K,j):(K,K) \to (K,R)$ is the identity $1_K$,

\noindent $(K,p):(K,R) \to (K,K)$ is $0$ (any morphism $K\to R$ initially factors through $j$),

\noindent $(K,\epsilon):(K,R) \to (K,R)$ is therefore $0$.

Recall that the projective objects of the category of finitely presented functors are exactly the representable functors (and the injective objects are exactly those of the form $A\otimes_R-$ where $A$ is a finitely presented {\em right} module), see for example \cite[10.1.14, 12.1.13]{PreNBK}.

The morphism $p$ gives the exact sequence of functors $$0 \to (K,-) \xrightarrow{(p,-)} (R,-) \to F_p \to 0;$$ so $F_pM =M/{\rm ann}_M(\epsilon)$ (noting that $(K,-)$ takes a module to the subspace of its elements which annihilate $\epsilon$).  The functor $F_p$ is nonzero only on one indecomposable module, namely $R$, where its value is 1-dimensional, so it is a simple functor, which we will denote by $S$.

We also have the induced exact sequence $$0 \to (K,-) \xrightarrow{(p,-)}(R,-) \xrightarrow {(j,-)}(K,-) \to F_j \to 0$$ (note that $(j,-)$ factors as the projection to $F_p$ followed by the inclusion of that in $(K,-)$).  The functor $F_j$ takes a module $M$ to ${\rm ann}_M(\epsilon)/\epsilon M$.  Again, this is nonzero only on the one indecomposable module $K$ and has 1-dimensional image there, so it is simple, and not isomorphic to $S$; we denote it $T$.

Thus we have the projective presentations $$0 \to (K,-) \xrightarrow{(p,-)} (R,-) \xrightarrow{\pi_S} S \to 0$$ and $$0 \to (K,-) \xrightarrow{(p,-)}(R,-) \xrightarrow {(j,-)}(K,-) \xrightarrow{\pi_T} T \to 0$$ of these simple functors and hence the exact sequence $$0\to S \to (K,-) \to T \to 0.$$

By general Auslander-Reiten theory (e.g.~\cite[IV.6.8]{AsSiSk}) we know that the simple functors are in natural bijection with the indecomposable finite-length modules; therefore we have found all the simple functors; we continue to find all the indecomposable functors.

\vspace{4pt}

We know that $(p,-)$ embeds $(K,-)$ as the radical of $(R,-)$, with cokernel $S$, that $(j,-)$ maps $(R,-)$ to $(K,-)$ with cokernel $T$ and with kernel the copy, ${\rm im}(p,-)$, of $(K,-)$ inside $(R,-)$, and that both functors $(R,-)$ and $K,-)$ are local.  So we obtain $S$-above-$T$-above-$S$ as the (unique) composition series for the forgetful functor $(R,-)$.  Then we can see the following indecomposable functors:  $(R,-)= S$-above-$T$-above-$S$ (projective and injective); $S$-above-$T$ (injective); $(K,-) = T$-above-$S$ (projective); $S$, $T$.  That there are no more is seen by computing the Auslander-Reiten quiver of this category (which, as we know, is the category of modules over the Auslander algebra - which is the endomorphism ring of $R\oplus K$).  That AR-quiver is shown below (best drawn on a cylinder, with the dashed lines identified; dotted horizontal lines indicate Auslander-Reiten sequences).  The second version shows the (unique) composition series of each indecomposable.

$
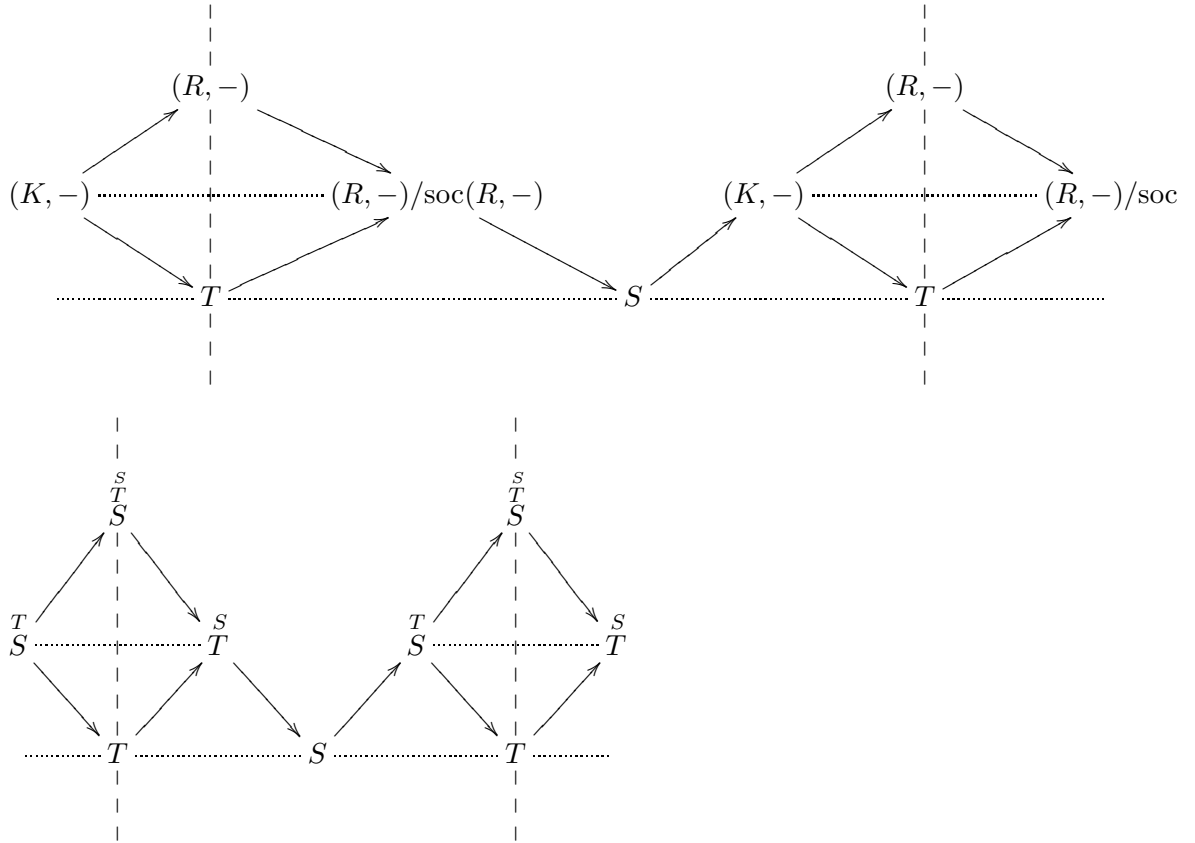
\xymatrix{ & & & & & \\ & (R,-) \ar[rd] \ar@{--}[u] \ar@{--}[dd] & & & & (R,-) \ar[rd] \ar@{--}[u] \ar@{--}[dd] \\ (K,-) \ar[ru] \ar[rd] \ar@{.}[rr] & & (R,-)/{\rm soc}(R,-) \ar[rd] & & (K,-) \ar[rd] \ar[ru] \ar@{.}[rr] & & (R,-)/{\rm soc} \\  \ar@{.}[r] & T \ar[ur] \ar@{.}[rr] \ar@{--}[d] & & S \ar[ur] \ar@{.}[rr] & & T \ar[ur]  \ar@{.}[r] \ar@{--}[d] & \\ & & & & & }$

$\xymatrix{ & & & & & \\ & \stackrel{\stackrel{S}{T}}{S} \ar[rd] \ar@{--}[u] \ar@{--}[dd] & & & & \stackrel{\stackrel{S}{T}}{S} \ar[rd] \ar@{--}[u] \ar@{--}[dd] \\ \stackrel{T}{S} \ar[ru] \ar[rd] \ar@{.}[rr] & & \stackrel{S}{T} \ar[rd] & & \stackrel{T}{S} \ar[rd] \ar[ru] \ar@{.}[rr] & & \stackrel{S}{T} \\  \ar@{.}[r] & T \ar[ur] \ar@{.}[rr] \ar@{--}[d] & & S \ar[ur] \ar@{.}[rr] & & T \ar[ur]  \ar@{.}[r] \ar@{--}[d] & \\ & & & & & }$

\noindent Each object of this category $(R\mbox{-}{\rm mod}, {\bf Ab})^{\rm fp}$ can be expressed by a pair of pp formulas and these already appeared implicitly in our computatations above:

\noindent $(R,-)$ is $(x=x)/(x=0)$, taking $M$ to $M$;

\noindent $(R,-)/{\rm soc}(R,-) $ is $(x=x)/(\epsilon | x)$, taking $M$ to $M/\epsilon M$;

\noindent $(K,-)$ is $(\epsilon x =0)/(x=0)$, taking $M$ to ${\rm ann}_M(\epsilon)$;

\noindent $S$ is $(\epsilon | x)/(x=0) \simeq (x=x)/(\epsilon x=0)$, taking $M$ to $\epsilon M \simeq M/{\rm ann}_M(\epsilon)$;

\noindent $T$ is $(\epsilon x=0)/(\epsilon | x)$, taking $M$ to ${\rm ann}_M(\epsilon)/\epsilon M$.

\noindent As in the example in Section \ref{secA3again}, one may compute a pp formula $\rho(x,x')$ defining any given morphism (i.e.~natural transformation) between these functors.  Also as in that example, all of them are rather simple, being built from $x'=x$ and $x'=0$.

\subsection{A localisation}

Let us compute the category ${\mathcal A}(_RR)$ of pp-pairs for the free module $_RR$, equivalently for the category of projective = flat $R$-modules, as a localisation of ${\rm Ab}(R)$.  We need to identify the Serre subcategory ${\mathcal S}_R$ of functors $F$ which are $0$ when evaluated at $R$.  Looking at, for example, the expression of each indecomposable functor as a pp-pair, it is immediate that only indecomposable in ${\mathcal S}_R$ is $T$ and hence that subcategory consists just of direct sums of copies of $T$.  What, then, is the quotient category ${\mathcal A}(R) = {\rm Ab}(R)/{\mathcal S}_R$?  Since $T$ becomes $0$ in the quotient category we see that this category has the following shape.

$
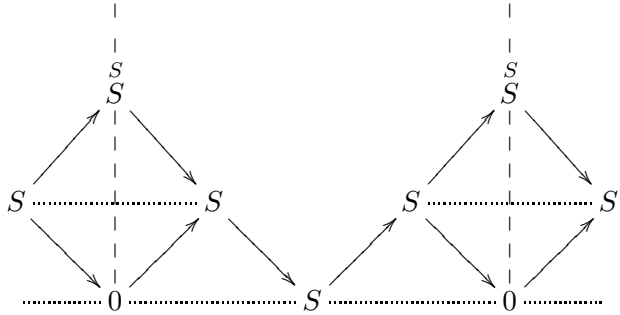
\xymatrix{ & & & & & \\ & \stackrel{S}{S} \ar[rd] \ar@{--}[u] \ar@{--}[dd] & & & & \stackrel{S}{S} \ar[rd] \ar@{--}[u] \ar@{--}[dd] \\ S \ar[ru] \ar[rd] \ar@{.}[rr] & & S \ar[rd] & & S \ar[rd] \ar[ru] \ar@{.}[rr] & & S \\  \ar@{.}[r] & 0 \ar[ur] \ar@{.}[rr] & & S \ar[ur] \ar@{.}[rr] & & 0 \ar[ur]  \ar@{.}[r]  & \\ }$

Looking at the middle two maps, an epimorphism, respectively monomorphism, in ${\rm Ab}(R)$, the kernel, resp.~cokernel, is now 0 in the quotient category so these have become isomorphisms in ${\mathcal A}(R)$ (we are using that the functor ${\rm Ab}(R) \to {\mathcal A}(R)$ is exact).  So we see that the quotient category is as follows.

$
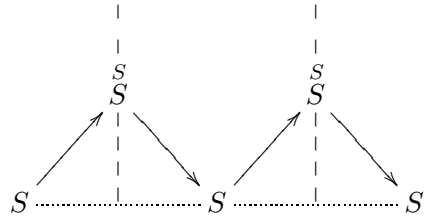
\xymatrix{ & & & & & \\ & & \stackrel{S}{S} \ar[rd] \ar@{--}[u] \ar@{--}[d] & & \stackrel{S}{S} \ar[rd] \ar@{--}[u] \ar@{--}[d] \\ & S \ar[ru] \ar@{.}[rr] & & S \ar[ru] \ar@{.}[rr] & & S}$

That is, ${\mathcal A}(R) \simeq R\mbox{-}{\rm mod}$.  In fact this is a general result, see \cite[7.1]{PreAxtFlat}.  There it is the statement that the definable subcategory generated by the flat modules is the category of exact functors from the smallest abelian subcategory of $R\mbox{-}{\rm Mod}$ containing $R\mbox{-}{\rm mod}$; by \ref{funD} this is an equivalent statement.  In our case, where the ring is coherent, the statement is simpler because the flat modules already form a definable subcategory and the finitely presented modules form an abelian subcategory of the category of all modules.

\subsection{Tensor product on ${\rm Ab}(K[\epsilon])$}

Now we compute the tensor product on this category ${\rm Ab}(R)$ which is induced by the usual tensor product on modules over $R$.

The general definition \cite{BVHP} of the induced monoidal/tensor structure on ${\rm Ab}(R)$ (\cite{BVHP}) is:

\noindent given $A,B\in R\mbox{-}{\rm mod}$, define $\otimes$ on the corresponding representable functors by $(A,-)\otimes (B,-) = (A\otimes B,-)$;

\noindent given morphisms $f:A\to A'$ and $g:B\to B'$ between finitely presented modules, define $(f,-)\otimes (g,-) = (f\otimes g, -): (A'\otimes B', -) \to (A\otimes B, -)$.

\noindent The tensor product constructed on ${\rm Ab}(R)$ will be required to be right exact, so that forces the rest of the construction.  Namely, a typical object of ${\rm Ab}(R)$ is the cokernel of a morphism between representables:  it has the form $F_f$ for some morphism $f:A\to B$, which yields the exact sequence $$(B,-) \xrightarrow{(f,-)} (A,-) \xrightarrow{\pi} F_f \to 0.$$  Therefore if $C\in R\mbox{-}{\rm mod}$ then the value of $(C,-)\otimes F_f$ is forced by requiring the sequence $$(C,-)\otimes (B,-) \to (C,-)\otimes (A,-) \xrightarrow{\pi} (C,-) \otimes F_f \to 0$$ to be exact.  That can then be repeated to compute the general case $F_g\otimes F_f$ (explicit computations follow).

\vspace{4pt}

First we need to compute $\otimes$ on $R\mbox{-}{\rm mod}$.  Of course, $R\otimes -$ is the identity functor on $R\mbox{-}{\rm mod}$, hence $(R,-)\otimes -$ is the identity functor on ${\rm Ab}(R)$ -  it is so on the full subcategory of projective = representable objects of ${\rm Ab}(R)$ therefore it is so, by right exactness of $\otimes$, on all of ${\rm Ab}(R)$.

As to $(K,-)\otimes-$, and noting that $K\otimes K=K$, its action is given on the subcategory of projectives by:

\noindent $(K,-)\otimes (K,-) = (K,-)$ and $(K,-)\otimes (R,-) = (K,-)$;

\noindent $(1_K,-)\otimes (p,-) = 1_{(K,-)}$ since $1_K\otimes p$ must be an epimorphism hence an isomorphism ($1_K$ denotes the identity map on $K$, not a generator of the module $K$);

\noindent $(1_K,-)\otimes (j,-) = 0$ by direct computation (of $1_K\otimes j$ on a generator of $K\otimes K$) or just since $1_K\otimes j$ cannot be monic, since $K$ is not flat, hence is $0$.

\vspace{4pt}

Consider the projective presentations  $$0\to (K,-) \xrightarrow{(p,-)} (R,-) \xrightarrow{\pi_S} S \to 0 \hspace{0.7cm} \mbox{ (1)}$$ $$0 \to (K,-) \xrightarrow{(p,-)} (R,-) \xrightarrow{(j,-)} (K,-) \xrightarrow{\pi_T} T \to 0 \hspace{0.7cm} \mbox{ (2)}$$  of $S$ and $T$.

\vspace{4pt}

Applying $K\otimes-$ to (2), we obtain the right exact sequence $$(K\otimes K,-) \xrightarrow{(1_K\otimes p,-)} (K\otimes R,-) \xrightarrow{(1_K\otimes j,-)} (K\otimes K,-) \xrightarrow{(1_K,-)\otimes \pi_T} (K,-)\otimes T \to 0$$  that is $$(K,-) \xrightarrow{1} (K,-) \xrightarrow{0} (K,-) \xrightarrow{\pi_T} (K,-)\otimes T \to 0,$$ from which we deduce $(K,-)\otimes T=(K,-)$, and $(1_K,-)\otimes \pi_T=1_{(K,-)}$ \hspace{0.2cm} (3).

\vspace{4pt}

Applying $K\otimes -$ to (1), we obtain $$(K\otimes K,-) \xrightarrow{(1_K\otimes p,-)} (K\otimes R,-) \xrightarrow{(1_K,-) \otimes \pi_S} (K,-)\otimes S \to 0,$$ that is, $(K,-) \xrightarrow{1} (K,-) \xrightarrow{(1_K,-) \otimes \pi_S} (K,-)\otimes S \to 0$, from which we deduce $(K,-)\otimes S =0$ and $(1_K,-)\otimes \pi_S =0$ \hspace{0.2cm} (4).

\vspace{4pt}

Applying $S\otimes -$ to (1), we obtain $$S\otimes (K,-) \xrightarrow{1_S\otimes (p,-)} S\otimes (R,-) \xrightarrow{1_S\otimes \pi_S} S\otimes S \to 0,$$ that is, using (4), $0 \to S\xrightarrow{1_S\otimes \pi_S} S\otimes S \to 0$.  So $S\otimes S =S$, $1_S\otimes \pi_S = 1_S$ \hspace{0.2cm} (5).

\vspace{4pt}

Applying $S\otimes -$ to (2), we obtain $$S\otimes (K,-) \xrightarrow{(1_S\otimes (p,-)} S\otimes (R,-) \xrightarrow{1_S\otimes (j,-)} S\otimes (K,-) \xrightarrow{1_S\otimes \pi_T} S\otimes T \to 0,$$ that is, using (4), $0 \to S \to 0 \xrightarrow{S\otimes \pi_T} S\otimes T \to 0$, from which we deduce $S\otimes T=0$ and $1_S\otimes \pi_T=0$ \hspace{0.2cm} (6).

\vspace{4pt}

Applying $T\otimes -$ to (1), we obtain $$T\otimes (K,-) \xrightarrow{1_T\otimes (p,-)} T\otimes (R,-) \xrightarrow{1_T\otimes \pi_S} T\otimes S \to 0,$$ that is, using (3) and (6),  $(K,-) \xrightarrow{1_T\otimes (p,-)} T \to 0 \to 0$, from which we deduce $1_T\otimes (p,-)$ is projection from $(K,-)$ to its top $T$ \hspace{0.2cm} (7).

\vspace{4pt}

Finally, applying $T\otimes -$ to (2), we obtain $$T\otimes (K,-) \xrightarrow{1_T\otimes (p,-)} T\otimes (R,-) \xrightarrow{1_T\otimes (j,-)} T\otimes (K,-) \xrightarrow{1_T\otimes \pi_T} T\otimes T \to 0,$$ that is, using (3), $(K,-) \xrightarrow{1_T\otimes (p,-)} T \xrightarrow{1_T\otimes (j,-)} (K,-) \xrightarrow{T\otimes \pi_T} T\otimes T \to 0$, from which we see $1_T\otimes (j,-) =0$ (there being only the zero map from $T$ to $(K,-)$) hence $T\otimes T =(K,-)$ and $1_T\otimes \pi_T =1_{(K,-)}$ \hspace{0.2cm} (8).

\vspace{4pt}

We have, in particular, that $S\otimes S=S$, $S\otimes T =0$, $T\otimes T = (K,-)$.

\vspace{4pt}

Now, $S$ is a subobject of a representable functor, hence (Burke, see\cite[10.2.14]{PreNBK}), it is a pp formula, indeed is $\epsilon|x$, that is $(\epsilon | x)/(x=0)$ as an object of this category of pp-pairs.  So we have just shown that $(\epsilon |x) \otimes (\epsilon |x)$ is again the pp formula $\epsilon |x$.  This is interesting, not least because it shows that the induced tensor product on pp-pairs (in particular on pp formulas) is not what one might at first guess it to be, see Section \ref{secpptensor}.

\vspace{4pt}

Lest these computations seem a little {\it ad hoc}, we recompute these values of $\otimes$ on ${\rm Ab}(R)$, directly from the definition of the extension of $\otimes$ to the functor category, using the projective presentations $$(K,-) \xrightarrow{(p,-)} (R,-) \xrightarrow{\pi_S} S \to 0$$ and $$(R,-) \xrightarrow{(j,-)} (K,-) \xrightarrow{\pi_T} T \to 0$$ of the simple functors $S$ and $T$.

\vspace{4pt}

To compute $S\otimes S$:

$\xymatrix{
(K\otimes K,-) \ar[rr]^{(p\otimes 1_K,-)} \ar[d]_{(1_K \otimes p,-)} && (R\otimes K,-) \ar[rr]^{\pi_S \otimes (1_K,-)} \ar[d]_{(1_R\otimes p,-)} && S\otimes (K,-)  \ar[d]_{1_S\otimes (p,-)} \ar[r] & 0 \\
(K\otimes R,-) \ar[rr]^{(p\otimes 1_R,-)} \ar[d]_{(1_K,-)\otimes \pi_S} && (R\otimes R,-) \ar[rr]^{\pi_S\otimes (1_R,-)} \ar[d]_{(1_R,-)\otimes \pi_S} && S\otimes (R,-)  \ar[d]_{1_S\otimes \pi_S} \ar[r] & 0 \\
(K,-)\otimes S \ar[rr]^{(p,-)\otimes 1_S} \ar[d] && (R,-)\otimes S \ar[rr]^{\pi_S\otimes 1_S} \ar[d] && S\otimes S  \ar[r] \ar[d] \ar[r] & 0 \\
0 && 0 && 0
}$

\noindent Which simplifies just using the values of $\otimes$ on the subcategory of representables and the fact that $(R,-)\otimes -$ is the identity, together with right exactness of $\otimes$, to:

$\xymatrix{(K,-)\ar[r]^{1} \ar[d]_{1} & (K,-) \ar[r] \ar[d]_{(p,-)} & 0  \ar[d] \ar[r] & 0 \\
(K,-) \ar[r]^{(p,-)} \ar[d] & (R,-) \ar[r]^{\pi_S} \ar[d]_{\pi_S} & S  \ar[d]^{\pi_S\otimes 1_S} \ar[r] & 0 \\
0 \ar[r] \ar[d] & S \ar[r]^{\pi_S\otimes 1_S} \ar[d] & S\otimes S  \ar[r] \ar[d] \ar[r] & 0 \\
0 & 0 & 0
}$

\noindent Hence $S\otimes S =S$ and $\pi_S\otimes 1_S=1_S$.

\vspace{4pt}

To compute $T\otimes T$:

$\xymatrix{
(R\otimes R,-) \ar[rr]^{(j\otimes 1_R,-)} \ar[d]_{(1_R \otimes j,-)} && (K\otimes R,-) \ar[rr]^{\pi_T \otimes (1_R,-)} \ar[d]_{(1_K\otimes j,-)} && T\otimes (R,-)  \ar[d]_{1_T\otimes (j,-)} \ar[r] & 0 \\
(R\otimes K,-) \ar[rr]^{(j\otimes 1_K,-)} \ar[d]_{(1_R,-)\otimes \pi_T} && (K\otimes K,-) \ar[rr]^{\pi_T\otimes (1_K,-)} \ar[d]_{(1_K,-)\otimes \pi_T} && T\otimes (K,-)  \ar[d]_{1_T\otimes \pi_T} \ar[r] & 0 \\
(R,-)\otimes T \ar[rr]^{(j,-)\otimes 1_T} \ar[d] && (K,-)\otimes T \ar[rr]^{\pi_T\otimes 1_T} \ar[d] && T\otimes T  \ar[r] \ar[d] \ar[r] & 0 \\
0 && 0 && 0
}$

\noindent Which simplifies to:

$\xymatrix{(R,-)\ar[r]^{(j,-)} \ar[d]_{(j,-)} & (K,-) \ar[r]^{\pi_T} \ar[d]_{0} & T  \ar[d]^{1_T\otimes (j,-)} \ar[r] & 0 \\
(K,-) \ar[r]^{0} \ar[d]_{\pi_T} & (K,-) \ar[r]^{1} \ar[d]_{1} & (K,-)  \ar[d]^{1_T\otimes \pi_T} \ar[r] & 0 \\
T \ar[r]^{(j,-)\otimes 1_T} \ar[d] & (K,-) \ar[r]^{\pi_T\otimes 1_T} \ar[d] & T\otimes T  \ar[r] \ar[d] \ar[r] & 0 \\
0 & 0 & 0
}$

From commutativity of the diagram we must have $1_T\otimes (j,-)=0$, hence $T\otimes T \simeq (K,-)$ and $1_T\otimes \pi_T = 1_{(K,-)}$.

\vspace{4pt}

Finally, to compute $S\otimes T$:

$\xymatrix{
(K\otimes R,-) \ar[rr]^{(p\otimes 1_R,-)} \ar[d]_{(1_K \otimes j,-)} && (R\otimes R,-) \ar[rr]^{\pi_S \otimes (1_R,-)} \ar[d]_{(1_R\otimes j,-)} && S\otimes (R,-)  \ar[d]_{1_S\otimes (j,-)} \ar[r] & 0 \\
(K\otimes K,-) \ar[rr]^{(p\otimes 1_K,-)} \ar[d]_{(1_K,-)\otimes \pi_T} && (R\otimes K,-) \ar[rr]^{\pi_S\otimes (1_K,-)} \ar[d]_{(1_R,-)\otimes \pi_T} && S\otimes (K,-)  \ar[d]_{1_S\otimes \pi_T} \ar[r] & 0 \\
(K,-)\otimes T \ar[rr]^{(p,-)\otimes 1_T} \ar[d] && (R,-)\otimes T \ar[rr]^{\pi_S\otimes 1_T} \ar[d] && S\otimes T  \ar[r] \ar[d] \ar[r] & 0 \\
0 && 0 && 0
}$

\noindent Which simplifies to:

$\xymatrix{(K,-)\ar[r]^{(p,-)} \ar[d]_{0} & (R,-) \ar[r]^{\pi_S} \ar[d]_{(j,-)} & S  \ar[d]^{1_S\otimes (j,-)} \ar[r] & 0 \\
(K,-) \ar[r]^{1} \ar[d]_{(1_K,-)\otimes \pi_T} & (K,-) \ar[r]^{\pi_T} \ar[d]_{\pi_S} & S\otimes (K,-)  \ar[d]^{\pi_S\otimes 1_S} \ar[r] & 0 \\
(K,-)\otimes T \ar[r] \ar[d] & T \ar[r]^{\pi_S\otimes 1_S} \ar[d] & S\otimes T  \ar[r] \ar[d] \ar[r] & 0 \\
0 & 0 & 0
}$

From which we see that $S\otimes (K,-)=0$ and hence that $S\otimes T=0$.

\vspace{4pt}

Either method will give the remaining values of $\otimes$; for example as follows.

\vspace{4pt}

Consider the exact sequence $$0\to S \to \stackrel{T}{S}=(K,-) \to T \to 0$$ in ${\rm Ab}(R)$.

\vspace{4pt}

Tensor with $T$ to get $$T\otimes S \to T\otimes (K,-) \to T\otimes T \to 0,$$ that is, $0 \to T\otimes (K,-) \to (K,-) \to 0$.  So $T\otimes (K,-) = (K,-)$.

\vspace{4pt}

Tensor with $(K,-)$ which is $\otimes$-flat (see \cite{BVHP}), to get the exact sequence $$0\to (K,-)\otimes S \to (K,-)\otimes (K,-) \to (K,-)\otimes T \to 0,$$ that is, $0\to (K,-)\otimes S \to  (K,-) \to (K,-)\to 0$.  So $(K,-)\otimes S =0$.

\vspace{4pt}

Next, consider the exact sequence $$0\to T \to \stackrel{S}{T} \to S \to 0.$$

\vspace{4pt}

Tensor with $(K,-)$ to get the exact sequence $$0\to (K,-)\otimes T \to (K,-)\otimes \stackrel{S}{T} \to (K,-)\otimes S \to 0,$$ that is, $$0\to \stackrel{T}{S} \to  (K,-) \to (K,-)\otimes \stackrel{S}{T} \to 0 \to 0.$$  So $(K,-) \otimes \stackrel{S}{T} = (K,-)$.

\vspace{4pt}

Tensor the same sequence with $S$ to get $$S\otimes T \to S\otimes (K,-) \to S\otimes S \to 0,$$ that is, $0 \to S\otimes (K,-) \to S \to 0$.  So $S\otimes (K,-) = S$.

\vspace{4pt}

Finally, consider the exact sequence $$0\to S \to \stackrel{\stackrel{S}{T}}{S} = (R,-) \to \stackrel{S}{T}  \to 0.$$

Tensor with $T$ to get $$T\otimes S=0 \to T \to T\otimes \stackrel{S}{T} \to 0.$$  So $T \otimes \stackrel{S}{T} = T$.

\vspace{4pt}

Tensor the same sequence with $\stackrel{S}{T}$ to get $$\stackrel{S}{T} \otimes S \to  \stackrel{S}{T} \to \stackrel{S}{T} \otimes \stackrel{S}{T} \to 0,$$ that is, $ S \to\stackrel{S}{T} \to \stackrel{S}{T} \otimes \stackrel{S}{T} \to 0$.  So $\stackrel{S}{T} \otimes \stackrel{S}{T} = \stackrel{S}{T}$.

\vspace{4pt}

In summary, here is the complete table of values of $\otimes$ on objects of ${\rm Ab}(R)$.

\begin{tabular}{c|ccccc}
& $S$ & $T$ & $(K,-)=\stackrel{T}{S}$ & $\stackrel{S}{T}$ & $(R,-)=\stackrel{\stackrel{S}{T}}{S}$ \\
\hline
$S$ & $S$ & $0$ & $0$ & $S$ & $S$ \\
\\
$T$ & $0$ & $(K,-)$ & $(K,-)$ & $T$ & $T$ \\
\\
$(K,-)$ & $0$ & $(K,-)$ & $(K,-)$ & $(K,-)$ & $(K,-)$ \\
\\
$\stackrel{S}{T}$ & $S$ & $T$ & $(K,-)$ & $\stackrel{S}{T}$ & $\stackrel{S}{T}$  \\
\\
$(R,-)$ & $S$ & $T$ & $(K,-)$ & $\stackrel{S}{T}$ & $(R,-)$ \\
\end{tabular}

\subsection{Tensor product of pp-pairs} \label{secpptensor}

Since the objects of the functor category ${\rm Ab}(R)$ can be seen as pp-pairs, it follows that there is a tensor product on these in the case that there is a tensor product on $R\mbox{-}{\rm mod}$, in particular one may make sense of the tensor product of two pp formulas for any commutative ring.  It is not, however, what one might at first guess it to be.  Suppose throughout this subsection that we have a given monoidal operation $\otimes$ on $R\mbox{-}{\rm mod}$.

One can do the following.  Let $\phi$ and $\psi$ be pp formulas for $R$-modules.  Choose corresponding free realisations $(C_\phi, \overline{c}_\phi)$ and $(C_{\phi'}, \overline{c}_{\phi'})$.  That is, $C_\phi$ is a finitely presented ${\mathcal C}$-module and $c_\phi$ an $n$-tuple of elements from $C_\phi$ with pp-type in $C_\phi$ generated by $\phi$, where $n$ is the number of free variables of $\phi$.  Similarly for $\psi$, which has, say, $n'$ free variables.  Then consider the $R$-module $C_\phi \otimes C_{\phi'}$ and the $n\times n'$-tuple $\overline{c_\phi} \otimes \overline{c_{\phi'}}$ in $C_\phi \otimes C_{\phi'}$.  Here $\overline{c_\phi} \otimes \overline{c_{\phi'}}$ means the tuple with entries $c_i\otimes c'_j$ (using hopefully clear notation).  In the case $n=n'=1$ this means $c_\phi\otimes c_{\phi'}$.  Since we are assuming that the tensor product of finitely presented modules is finitely presented, the pp-type of $\overline{c_\phi} \otimes \overline{c_{\phi'}}$ in $C_\phi \otimes C_{\phi'}$ can be generated by a pp formula (as opposed to an infinite set of pp formulas).  It is tempting to define $\phi\otimes \phi'$ to be any generator of that pp-type.  {\em But}, referring to the example above, namely $S\otimes S$ where $S =(\epsilon | x)$, a free realisation of $\epsilon|x$ is $(R,\epsilon)$ and the element $\epsilon\otimes \epsilon$ in $R\otimes R=R$ is $0$, so this attempt at a definition would give the formula $x=0$, not the correct formula $\epsilon | x$.  (This recipe using free realisations does work for quantifier-free pp formulas - these correspond to the projective = representable functors - because in that case, the tuples chosen will (without loss of generality) generate the modules $C_\phi$ and $C_\psi$.)

\subsection{Another monoidal structure ${\rm Ab}(K[\epsilon])$ when ${\rm char}(K)=2$} \label{secgptensor}

Just to show that inducing tensor products on pp formulas is rather intriguing, we consider another monoidal structure on modules over $K[\epsilon: \epsilon^2=0]$ in the case that the characteristic of $K$ is $2$.  In that case, $R=K[\epsilon; \epsilon^2=0]$ is also the group algebra over $K$ of the cyclic group of order 2.  Namely, if $g$ is the non-identity element of that group then we set $\epsilon =1+g$.  The tensor product of $A,B\in R\mbox{-}{\rm mod}$  as group representations is obtained by forming $A\otimes_K B$ and then having $g$ act by $g.a\otimes b = ga\otimes gb$.  This is quite different from the tensor product as $R$-modules:  for example, this tensor product is exact, not just right exact, and one can compute, for example, that $R\otimes_K R \simeq R^2$, an $R$-basis for the latter being $1\otimes \epsilon, (1\otimes 1 + 1\otimes \epsilon)$.  We will write $\otimes_K$ in order to distinguish this from the tensor product over $R$.

We can repeat the same kinds of computations of the induced monoidal structure on the functor category as for the usual tensor product over $R$; we will see that the induced tensor product on the functor category is different from the previous one.

First we see that $K\otimes_KK =K$ and $R\otimes_KK \simeq R \simeq K\otimes_KR$.  The trivial module $K$ is a $\otimes_K$ unit, with $1_K\otimes_K j=j$ and $1_K\otimes_K p =p$, where $j$, $p$ are as in the exact sequence $$0\to K \xrightarrow{j} R \xrightarrow{p} K \to 0.$$  Exactness of $R\otimes_K-$ applied to that sequence gives the exact sequence $$0\to R \xrightarrow{1_R\otimes_K j} R^2 \xrightarrow{1_R\otimes_K p} R \to 0$$ which must be split-exact, with $1_R\otimes_K j$ and $1_R\otimes_K p$ being respectively a split embedding and split epimorphism.  We continue to use the notation from Section \ref{sectensor}.

To compute $S\otimes_KS$:

$\xymatrix{(K\otimes_K K,-) \ar[rr]^{(p\otimes_K 1_K,-)} \ar[d]_{(1_K\otimes_K p,-)} && (R\otimes_K K,-) \ar[rr]^{\pi_S \otimes_K (1_K,-)} \ar[d]_{(1_R\otimes_K p,-)} && S\otimes_K (K,-) \ar[d]_{1_S \otimes_K (p,-)} \ar[r] & 0 \\
(K\otimes_K R,-) \ar[rr]^{(p\otimes_K 1_R,-)} \ar[d]_{(1_K,-)\otimes_K \pi_S} && (R\otimes _K R,-) \ar[rr]^{\pi_S \otimes_K (1_R,-)} \ar[d]_{(1_R,-)\otimes_K \pi_S} && S \otimes_K (R,-)  \ar[d]_{ 1_S \otimes_K \pi_S} \ar[r] & 0 \\
(K,-) \otimes_K S \ar[rr]^{(p,-)\otimes_K 1_S} \ar[d] && (R,-)\otimes_K S \ar[rr]^{\pi_S \otimes_K 1_S} \ar[d] && S\otimes_K S  \ar[r] \ar[d] \ar[r] & 0 \\
0 && 0 && 0
}$

\noindent That is:

$\xymatrix{(K,-)\ar[rr]^{(p,-)} \ar[d]_{(p,-)} && (R,-) \ar[rr]^{\pi_S} \ar[d]_{(1_R\otimes_K p,-)} && S  \ar[d]_{1_S \otimes_K (p,-)} \ar[r] & 0 \\
(R,-) \ar[rr]^{(p\otimes_K 1_R,-)} \ar[d]_{\pi_S} && (R^2,-) \ar[rr]^{\pi_S \otimes_K (1_R,-)} \ar[d]_{(1_R,-)\otimes_K \pi_S} && (R,-)\otimes_K S  \ar[d]_{1_S\otimes_K \pi_S} \ar[r] & 0 \\
S \ar[rr]^{(p,-) \otimes_K 1_S} \ar[d] && S\otimes_K (R,-) \ar[rr]^{\pi_S\otimes_K 1_S} \ar[d] && S\otimes_K S \ar[r] \ar[d] \ar[r] & 0 \\
0 && 0 && 0
}$

\noindent Since $1_R\otimes_K p$ is a split epi, $(1_R\otimes_Kp,-)$ is a split embedding, so we deduce $(R,-)\otimes_KS =(R,-)$ and $(1_R,-)\otimes_K \pi_S$ is a split epi.  The question then is whether $1_S\otimes_K (p,-):S \to (R,-)\otimes_K S$ is or is not $0$, equivalently, since $\pi_S$ is epi, whether, in the top left square, the composition down then across is $0$, equivalently whether the image of $(1_R \otimes_K p,-)$ is the same as the image of $(p\otimes_K 1_R,-)$.  We can compute the components of these natural transformations at the $R$-module $R$ and see that the image of the first map is the vector subspace of $R\otimes_K R$ generated by $1_R\otimes_K\epsilon$ and $ \epsilon\otimes_K \epsilon$ (where here $\epsilon$ means the endomorphism multiplication-by-$\epsilon$ of $R$).  By symmetry, the image of the horizontal map is the subspace generated by $\epsilon \otimes_K 1_R$ and $\epsilon\otimes_K \epsilon$.  They are not the same, hence $1_S \otimes_K (p,-)$ is monic.  Therefore its cokernel, $S\otimes_K S = \stackrel{S}{T}$.  Recall from before that $S\otimes S =S$ under the monoidal structure on ${\rm Ab}(R)$ induced from tensoring over $R$, so these structures are quite different and, with respect to this tensor, we have that $(\epsilon | x) \otimes_K (\epsilon |x)$ is the formula $x=x$ (as opposed to $\epsilon |x$ for the other structure).

\vspace{4pt}

To compute $T\otimes T$ we start with the same $3\times 3$ diagram of right exact sequences as in Section \ref{sectensor} but with $\otimes_K$ replacing $\otimes$ and recalling that $(K,-)$ is the $\otimes_K$-identity on ${\rm Ab}(R)$.  That gives:

$\xymatrix{
(R^2,-) \ar[rr]^{(j\otimes_K 1_R,-)} \ar[d]_{(1_R \otimes_K j,-)} && (R,-) \ar[rr]^{\pi_T \otimes_K (1_R,-)} \ar[d]_{(j,-)} && T\otimes_K (R,-)  \ar[d]_{1_T\otimes_K (j,-)} \ar[r] & 0 \\
(R,-) \ar[rr]^{(j,-)} \ar[d]_{(1_R,-)\otimes_K \pi_T} && (K,-) \ar[rr]^{\pi_T} \ar[d]_{\pi_T} && T  \ar[d]_{1_T\otimes_K \pi_T} \ar[r] & 0 \\
(R,-)\otimes_K T \ar[rr]^{(j,-)\otimes_K 1_T} \ar[d] && T \ar[rr]^{\pi_T\otimes_K 1_T} \ar[d] && T\otimes_K T  \ar[r] \ar[d] \ar[r] & 0 \\
0 && 0 && 0
}$

\noindent Since $(j\otimes_K 1_R,-)$ is a split projection, $T\otimes_K (R,-)=0$ and $\pi_T \otimes_K (1_R,-) =0$

$\xymatrix{(R^2,-)\ar[r]^{(j\otimes_K 1_R,-)} \ar[d]_{(1_R\otimes_K j,-)} & (R,-) \ar[r]^{0} \ar[d]_{(j,-)} & 0  \ar[d] \ar[r] & 0 \\
(R,-) \ar[r]^{(j,-)} \ar[d]_{\pi_T} & (K,-) \ar[r]^{\pi_T} \ar[d]_{\pi_T} & T  \ar[d]^{1_T\otimes_K \pi_T} \ar[r] & 0 \\
0 \ar[r] \ar[d] & T \ar[r]^{\pi_T\otimes_K 1_T} \ar[d] & T\otimes_K T  \ar[r] \ar[d] \ar[r] & 0 \\
0 & 0 & 0
}$

\noindent From which we see that $T\otimes_K T = T$ and $1_T\otimes_K \pi_T = 1_T$.

\vspace{4pt}

Finally, to compute $S\otimes_K T$ we obtain the diagram:

$\xymatrix{
(R,-) \ar[rr]^{(p\otimes_K 1_R,-)} \ar[d]_{(j,-)} && (R^2,-) \ar[rr]^{\pi_S \otimes_K (1_R,-)} \ar[d]_{(1_R\otimes_K j,-)} && S\otimes_K (R,-)  \ar[d]_{1_S\otimes_K (j,-)} \ar[r] & 0 \\
(K,-) \ar[rr]^{(p,-)} \ar[d]_{\pi_T} && (R,-) \ar[rr]^{\pi_S} \ar[d]_{(1_R,-)\otimes_K \pi_T} && S  \ar[d]_{1_S\otimes_K \pi_T} \ar[r] & 0 \\
T \ar[rr]^{(p,-)\otimes_K 1_T} \ar[d] && (R,-)\otimes_K T \ar[rr]^{\pi_S\otimes_K 1_T} \ar[d] && S\otimes_K T  \ar[r] \ar[d] \ar[r] & 0 \\
0 && 0 && 0
}$

\noindent We see (again) that $(R,-)\otimes_K T=0$ and hence deduce $S\otimes_K T=0$.

\vspace{4pt}

In summary:  $S\otimes_K S =(R,-)$; $T\otimes_K T=T$; $S\otimes_K T=0$.

Indeed, further computation (given after the table) gives us the following.

\begin{tabular}{c|ccccc}
$\otimes_K$ & $S$ & $T$ & $(K,-)=\stackrel{T}{S}$ & $\stackrel{S}{T}$ & $(R,-)=\stackrel{\stackrel{S}{T}}{S}$ \\
\hline
$S$ & $\stackrel{S}{T}$ & $0$ & $S$ & $\stackrel{S}{T}$ & $(R,-)$ \\
\\
$T$ & $0$ & $T$ & $T$ & $0$ & $0$ \\
\\
$(K,-)$ & $S$ & $T$ & $(K,-)$ & $\stackrel{S}{T}$ & $(R,-)$ \\
\\
$\stackrel{S}{T}$ & $\stackrel{S}{T}$ & $0$ & $\stackrel{S}{T}$ & $\stackrel{S}{T}$ & $(R,-)$  \\
\\
$(R,-)$ & $(R,-)$ & $0$ & $(R,-)$ & $(R,-)$ & $(R,-)^2$ \\
\end{tabular}

\vspace{6pt}

We can get the other entries from the exact sequence $$0\to T \to \stackrel{S}{T} \to S \to 0.$$

Tensoring with $S$ gives the exact sequence $$ S\otimes_K T \to S\otimes_K\stackrel{S}{T} \to S\otimes_KS \to 0,$$  that is $ 0 \to S\otimes_K \stackrel{S}{T} \to \stackrel{S}{T}\to 0$ from which we deduce $S\otimes_K \stackrel{S}{T} = \stackrel{S}{T}$.

Next, since $ (R,-) \to \stackrel{S}{T} $ is an epimorphism, so is $T\otimes_K (R,-) \to T\otimes_K \stackrel{S}{T} $ so, since $T\otimes_K (R,-) =0$ we have $T\otimes_K \stackrel{S}{T} =0$.

Then, tensoring the first exact sequence with $\stackrel{S}{T}$ gives $$\stackrel{S}{T} \otimes_K T \to  \stackrel{S}{T}\otimes_K \stackrel{S}{T} \to \stackrel{S}{T} \otimes_K S \to 0,$$ that is, $0 \to \stackrel{S}{T} \otimes_K \stackrel{S}{T} \to \stackrel{S}{T} \to 0$, hence $\stackrel{S}{T}\otimes_K \stackrel{S}{T} = \stackrel{S}{T}$.

Finally, applying $(R,-)$ to the first exact sequence gives $$(R,-) \otimes_K T \to (R,-) \otimes_K \stackrel{S}{T} \to (R,-)\otimes_K S \to 0,$$ that is, $0 \to (R,-) \otimes \stackrel{S}{T} \to (R,-) \to 0$, so $(R,-) \otimes_K \stackrel{S}{T} = (R,-)$.

\end{document}